\documentclass[11pt]{amsart}
\usepackage{gensymb}
\usepackage{graphicx}
\usepackage{amsfonts}
\usepackage{mathtools}
\usepackage{amsthm}
\usepackage{caption}
\usepackage{subcaption}
\usepackage[shortlabels]{enumitem}
\usepackage[hidelinks]{hyperref}
\usepackage[capitalise]{cleveref}
\newtheorem{theorem}{Theorem}
\newtheorem{lemma}[theorem]{Lemma}
\usepackage{amsmath}
\usepackage{amssymb}
\usepackage{bm}
\usepackage{xcolor}
\usepackage{tikz-cd}
\usepackage{tikz}
\newtheorem{claim}[theorem]{Claim}

\newtheorem{proposition}[theorem]{Proposition}

\usepackage{verbatim}
\usepackage{url}
\theoremstyle{definition}
\newtheorem{remark}[theorem]{Remark}
\newtheorem{definition}[theorem]{Definition}
\newtheorem{example}[theorem]{Example}

\usepackage{float}
\usepackage{multirow}
\usepackage[style=alphabetic, citestyle=alphabetic, maxbibnames=99,maxcitenames=99]{biblatex}
\usetikzlibrary{matrix,calc,arrows}

\usepackage{microtype}

\makeatletter\g@addto@macro\@verbatim{\microtypesetup{activate=false}}\makeatother%
\makeatletter

\title[Cohomology of the Moduli of Differentials ]{Cohomology of Moduli Space of Multiscale Differentials in genus 0}
\author{Prabhat Devkota}
\begin{document}
\maketitle
\begin{abstract}
  We prove that the rational cohomology ring of moduli space of multiscale differentials in genus 0 is generated by the classes of boundary divisors. The main idea is the technique of the Chow-K\"unneth generation Property and the observation that the intersection of a collection of boundary divisors in the moduli space is irreducible. We observe that the relations between the boundary strata in cohomology are generated by the pullback of the WDVV relations and the relations between the torus-invariant subvarieties in the fiber over $\overline{M}_{0,n}$. We also characterize the cases in which the moduli space is a smooth variety, and in these cases, we prove that the integral cohomology ring is generated by the boundary divisors.
\end{abstract}
\section*{Introduction}
Given an $n$-tuple of integers $\mu=(m_1,\ldots,m_n)$ with $\sum_im_i=2g-2$, the moduli space $\Omega\mathcal{M}_{g,n}(\mu)$ parametrizes abelian differentials $\omega$ on an $n$-pointed compact Riemann surface $(C;p_1,\ldots,p_n)$ of genus $g$ with prescribed zeros (or poles) of order $m_i$ at $p_i$. This moduli space admits a $\mathbb{C}^*$ action via rescaling the differentials; the quotient, which will also be called the moduli space of abelian differentials, denoted by $\mathcal{B}\coloneqq\mathbb{P}\Omega\mathcal{M}_{g,n}(\mu)$, admits a modular compactification $\mathbb{P}\Xi\overline{\mathcal{M}}_{g,n}(\mu)$ which is a smooth Deligne-Mumford stack with normal crossing boundary divisor, parametrizing ``multiscale differentials'' on stable nodal curves of arithmetic genus $g$ (see \cite{bcggm}). Denote by $\overline{\mathcal{B}}\coloneq \mathbb{P}\Xi\overline{\mathcal{M}}_{g,n}(\mu)$ this moduli stack of multiscale differentials. Its coarse moduli space, denoted $\overline{B}$, is proven in \cite{ccm22} to be a projective variety (\cite{cghms22} gives another proof of the projectivity). Our objective is to compute the cohomology with rational coefficients of the moduli space of multiscale differentials of genus zero:
\begin{theorem}\label{main-thm}
  The rational Chow and cohomology rings of the moduli space $\overline{B}$ of multiscale differentials in genus zero are isomorphic and generated by the boundary divisors.
\end{theorem}
One technique we use is the \textit{Chow-K\"unneth generation Property (CKgP)}. This property has been recently used in \cite{cl22} and \cite{clp23} for $\overline{M}_{g,n}$ with considerable success, and in this article, we further explore the geometry of the multiscale compactification and adapt the Chow-K\"unneth generation Property to the case of the moduli space of abelian differentials in genus 0.

In the case $\mu=(0^{n-1},-2)$, the moduli space of multiscale differentials is a smooth projective {\em variety}. In this case, we will prove the stronger statement for the {\em integral} cohomology: the Chow ring and cohomology ring with integral coefficients are isomorphic and are generated by boundary divisors. The CKgP does not work for integral cohomology, so we use a different method. The idea is to factor the birational morphism $\overline{B}\rightarrow\overline{M}_{0,n}$ into a composition of blowups along smooth irreducible centers and then use the formula for the cohomology (or Chow ring) of a regular blowup. The idea is motivated by the approach used in \cite{keel} to determine the cohomology ring of $\overline{M}_{0,n}$.

Besides $\mu=(0^{n-1},-2)$, there are other cases when $\overline{B}$ is a smooth variety. When $n\geq 7$, we have:
\begin{proposition}
  If $n\geq 7$ then $\overline{B}$ is a smooth variety if and only if $\mu=(0^{n-1},-2)$ or $(0^{n-2},-1^2)$.
\end{proposition}
 However, for $n=5,6$, there are in fact additional exceptional profiles $\mu$ for which the coarse moduli space remains a smooth variety; see \cref{smooth_spaces} for a complete list of smooth cases (of course, for $n\leq 4$, $\overline{B}\cong\overline{M}_{0,n}$ is always a smooth variety). The ideas involved in the computation of the integral cohomology for $\mu=(0^{n-1},-2)$ should also apply in these other smooth cases to allow us to conclude that the integral cohomology ring (as well as integral Chow ring) is generated by boundary divisors.\\

 Using the birational morphism $\pi:\overline{B}\rightarrow\overline{M}_{0,n}$, we can already understand the rational Picard group of $\overline{B}$. Indeed, the interior of $\overline{B}$ is equal to $M_{0,n}$ which is a Zariski open subset of $\mathbb{C}^{n-3}$ and so has trivial Picard group. So, $\text{Pic}(\overline{B})\otimes\mathbb{Q}$ is generated by the boundary divisors. The exceptional divisors of $\pi$ are linearly independent with each other (and with the pullback of the Picard group of $\overline{M}_{0,n}$), so the rank of the Picard group of $\overline{B}$ is equal to the rank of $\text{Pic}(\overline{M}_{0,n})$ plus the number of divisors of $\overline{B}$ exceptional over $\overline{M}_{0,n}$. In particular, the linear relations between the boundary divisors of $\overline{B}$ are generated by the pullback of the linear relations between the boundary divisors of $\overline{M}_{0,n}$. On the other hand, while the pullback of the Witten-Dijkgraaf-Verlinde-Verlinde (WDVV) relations are not sufficient to generate all the relations between the higher codimension boundary strata in $\overline{B}$, we will prove:
 \begin{theorem}\label{relations}
   The vector space of relations between the codimension $p$ boundary strata in the cohomology group $H^{2p}(\overline{B},\mathbb{Q})$ is $\mathbb{Q}$-linearly generated by the pullback of the WDVV relations from $\overline{M}_{0,n}$ and relations between the torus invariant subvarieties in an irreducible component of a fiber of the birational morphism $\pi:\overline{B}\rightarrow\overline{M}_{0,n}$.
 \end{theorem}

We note that while \cite{cghms22} describes the birational morphism $\pi:\overline{B}\rightarrow\overline{M}_{0,n}$ as a blowup of an explicit ideal sheaf on $\overline{M}_{0,n}$, it is inadequate for the computation of cohomology because the blowup locus defined by the ideal sheaf is neither irreducible nor reduced (and the components are not equidimensional). Consequently, the usual formula for computation of the cohomology ring of a regular blowup is not applicable.
 
 The outline of this paper is as follows. In \cref{prelim}, we will summarize the basic facts about the structure of the moduli space of multiscale differentials, that will be relevant throughout the paper, and also collect some consequences of Chow-K\"unneth generation Property that will be relevant in our argument. In \cref{irred_bdry}, we will prove that an intersection of a collection of boundary divisors, if non-empty, is irreducible. This fact will be crucial to prove that the boundary divisors will generate the cohomology ring, once we know that the cohomology ring is tautological. In \cref{rational_cohomology}, we will prove \cref{main-thm} as well as \cref{relations}. Then in \cref{smooth_spaces}, we will determine all the cases for which the coarse moduli space $\overline{B}$ is a smooth variety. That for the smooth cases, the integral cohomology is generated by the boundary divisors will be discussed in \cref{integral_cohomology}.
 
\subsection*{Acknowledgements}
I would like to thank my advisor Samuel Grushevsky for continuous encouragement and helpful advice as well as regular fruitful discussions. I would also like to thank Samir Canning, Myeongjae Lee and Martin M\"oller for fruitful discussions. Especially utmost thanks to Samir Canning for providing ideas to determine the relations between the boundary strata in \cref{relations_cohomology} and to Myeongjae Lee for providing the idea of the proof of \cref{prongs_unique}.
\section{Preliminaries}\label{prelim}
In this section, we collect some general facts about the moduli space of multiscale differentials, especially in regards to its boundary, as well as a review of the Chow-K\"unneth generation Property (CKgP) that will be used in this paper. For more details on the moduli space of multiscale differentials and the structure of its boundary, we refer to \cite{bcggm} and \cite{cmz22}, whereas for more details on the CKgP, we refer to \cite{cl22} and \cite{clp23}.
\subsection{Enhanced level graphs and multiscale differentials}
Given a point in the boundary $\partial\overline{B}\coloneq\overline{B}\setminus B$, there is an enhanced level graph compatible with that point (and the boundary is stratified by strata corresponding to such enhanced level graphs). An enhanced level graph $(\Gamma,\ell,\{\kappa_e\})$ comprises the data of:
\begin{itemize}
\item A connected stable graph $\Gamma=(V,E,H)$, with vertices $V$, edges $E$ and half edges $H$. Each vertex is further assigned with a non-negative integer, called the genus of the vertex. The vertices $V$ correspond to the irreducible components of the stable curve $C$, the edges $E$ correspond to the nodes of the curve $C$ and the half edges $H$ either correspond to the marked points (such half edges are also called legs) or pair up with some other half edge at an adjacent vertex to form an edge. The stability of the graph then corresponds to the stability of the pointed curve $(C;p_1,\ldots,p_n)$.
\item A total ordering on the vertices $V$ of $\Gamma$, where equality is permitted. This gives us the level structure on $\Gamma$. We can encode this by a level function $\ell:V\rightarrow\mathbb{Z}$ and for convenience, we assume that the image of this function is a set of consecutive integers with the largest (top level) being 0.
\item An assignment of a non-negative integer $\kappa_e$ to each edge of $\Gamma$, such that $\kappa_e=0$ if and only if the edge connects two vertices on the same level (such an edge is called a horizontal edge).
\end{itemize}
For brevity, we will frequently denote the enhanced level graph simply by $\Gamma$ whenever the confusion with the underlying dual graph does not arise. By a {\em leaf} of a dual graph, we will refer to an extremal vertex -- that is, a vertex of valency one. Note that, if a dual graph is a tree, it will always have a leaf. Similarly, by a \textit{rooted level tree}, we will refer to a level graph (whose dual graph is a tree) with a unique top or bottom level vertex. In this case, the unique vertex in the top (or the bottom) level will be called the \textit{root}.\\

Next, we discuss the twisted differentials and what it means for them to be compatible with an enhanced level graph. A twisted differential of type $\mu=(m_1,\ldots,m_n)$ on a stable curve $C$ with dual graph $\Gamma$ compatible with an enhanced level structure $(\Gamma,\ell,\{\kappa_e\})$ is a collection of differentials $\{\omega_v\}_{v\in V(\Gamma)}$ (where $\omega_v$ is a meromorphic differential on a component $C_v$ of $C$) with zeros and poles of order as prescribed by $\mu$ and the enhancements. To elucidate further, the differential $\omega_v$ will have a zero/pole on the half edges. If the half edge represents a marked point $p_i$ then $\omega_v$ has zero/pole of order $m_i$ at $p_i$. If the half edge pairs with another to form a horizontal edge $e$ then $\omega_v$ has a simple pole at the point in $C_v$ corresponding to the half edge, with residue opposite to that of the half edge it is paired with. And if the half edge pairs with another to form a vertical edge $e$, and the vertex $v$ lies on the lower level, then $\omega_v$ has a pole of order $-\kappa_e-1$ at the corresponding point in $C_v$, whereas if the vertex lies on the upper level then it has a zero of order $\kappa_e-1$. Furthermore, the differentials also satisfy an additional condition called the Global Residue Condition (GRC for brevity). Since GRC is automatically satisfied in genus 0, we simply refer to \cite{bcggm} for the details.

If a meromorphic differential $\omega$ on a curve $C$ (say smooth and irreducible) has a zero of order $m\geq 0$ at a point $p$ then locally around $p$, there exists a local coordinate $z$ in which it can be written as $\omega=z^mdz$. In this local chart, we have $\kappa=m+1$ outgoing horizontal prongs $\zeta^i_\kappa\frac{\partial}{\partial z}$ (where $\zeta_\kappa$ is the $\kappa$-th primitive root of unity). Analogously, when $\omega$ has a pole of order $m\leq -2$ at $p$, we have $\kappa=-m-1$ incoming horizontal prongs $-\zeta^i_\kappa\frac{\partial}{\partial z}$, where $z$ is a standard coordinate around $p$ in which the differential can be written as $\omega=(z^{m+1}+r)\frac{dz}{z}$ (where $r=\text{Res}_p\omega$). If $\omega$ is a twisted differential on a stable curve $(C;p_1,\ldots,p_n)$ and a vertical edge $e$ with enhancement $\kappa_e$ connects two components $C_1$ and $C_2$ (with $C_1$ at higher level), for the $\omega$ to yield a well-defined multiscale differential, we need to have additional datum of ``prongs-matchings'' between the set of horizontal prongs of $C_1$ and $C_2$. A prong-matching is a cyclic order reversing bijection between the outgoing prongs of $C_1$ and incoming prongs of $C_2$.

Thus, a multiscale differential of type $\mu=(m_1,\ldots,m_n)$ on a stable curve $(C;p_1,\ldots,p_n)$ comprises of the following data:
\begin{itemize}
\item An enhanced level structure on the dual graph $\Gamma$ of $(C;p_1,\ldots,p_n)$.
\item A collection of differentials $\{\omega_v\}_{v\in V(\Gamma)}$, compatible with the enhanced level structure of $\Gamma$.
\item Prong-matchings on each vertical edge.
\end{itemize}

The multiscale differentials only retain information on the lower level up to projectivization, so we have to consider the action by rescaling level-by-level of a multiplicative torus, called the level rotation torus $T_\Gamma$, isogeneous to $(\mathbb{C}^*)^{L(\Gamma)}$, on the space of twisted differentials compatible with the enhanced level graph $\Gamma$ (here, $L(\Gamma)$ is the number of levels below the top). This action also has the effect of twisting the prong matchings on the vertical edges. The moduli space of multiscale differentials then parametrizes the equivalence classes of multiscale differentials under such action. We refer to \cite{bcggm} or \cite{cmz22} for details and recap what we will need about the equivalence classes of prong matchings:

Given an enhanced level graph $(\Gamma,\ell,\{\kappa_e\})$, there are altogether $\prod_e\kappa_e$ possible choices of prong matchings on a compatible multiscale differential. Then understanding the equivalence classes of prong matchings amounts to understanding the equivalence classes under the rotation of the prongs by the level rotation group $R_\Gamma\cong \mathbb{Z}^{L(\Gamma)}$ whose $i^{\text{th}}$ factor acts on the $i^{\text{th}}$ level passage (that is, all the edges that cross the virtual horizontal line right above the level $-i$) by diagonally turning in a fixed direction the the prong-matchings at each edge crossing this level passage.

We will frequently use the term ``codimension of a level graph'' to mean the number of levels below the top (that is, the number of level passages) plus the number of horizontal edges. Indeed, the codimension of a level graph coincides with the codimension of the associated boundary stratum. In particular, a level graph of codimension one is called ``divisorial''. We will also call a boundary stratum given by a divisorial level graph with only one horizontal edge a ``horizontal divisor'' and those with no horizontal edge a ``vertical divisor''.

\begin{remark}\label{enhancements_on_tree}
  In genus zero (or more generally if the dual graph $\Gamma$ is a tree), the dual graph determines uniquely the enhancements on all the edges (but, in general, the level structure is not fully determined -- we can only recover a partial order on the vertices of the dual graph). Indeed, using the fact that the degree of a meromorphic differential on each component of genus $g_v$ is $2g_v-2$, we can start by determining the enhancement on an edge connecting to a leaf of the tree, and then use induction on the number of vertices.
  \end{remark}

\subsection{The structure of the boundary of the moduli space of multiscale differentials}
As mentioned before, the boundary $\partial\overline{B}$ of $\overline{B}$ is stratified into the loci of multiscale differentials compatible with various enhanced level graphs. Suppose $\Lambda$ is an enhanced level graph with $L$ level passages and denote by $D_\Lambda\subset\overline{B}$ the closure of the locus of multiscale differentials compatible with $\Lambda$. Then for each level $i$, the tuple $(\bm{g}^{[i]},\bm{n}^{[i]},\bm{\mu}^{[i]})$ consisting of genera of the components at level $i$, the marked points (legs) at level $i$ and the orders of zeros and poles on the half edges at level $i$, along with the residue conditions $\bm{\mathfrak{R}}^{[i]}$ imposed on some of the half legs, defines the generalized stratum $B_\Lambda^{[i]}\coloneqq \mathbb{P}\Xi\overline{\mathcal{M}}_{\bm{g}^{[i]},\bm{n}^{[i]}}^{\bm{\mathfrak{R}}^{[i]}}(\bm{\mu}^{[i]})$. We refer to \cite[sec. 4]{cmz22}, for details on the generalized strata. What will be important for us in the further discussion is the following:
\begin{proposition}\cite[Prop. 4.4]{cmz22}\label{product_stratum}
  There exists a stack $D_\Lambda^s$, called the simple boundary stratum of type $\Lambda$, that admits finite morphisms $c_\Lambda:D_\Lambda^s\rightarrow D_\Lambda$ and $p_\Lambda:D_\Lambda^s\rightarrow B_\Lambda\coloneqq\prod_iB_\Lambda^{[i]}$.
\end{proposition}
\begin{remark}\label{degree_product_stratum}
  In general, the degrees of these morphisms $c_\Lambda$ and $p_\Lambda$ are not easy to write down -- only the ratio of their degrees is given in \cite[Lemma 4.5]{cmz22}. However, over the open boundary stratum $U_\Lambda\subset D_\Lambda$ (consisting of differentials compatible with $\Lambda$ but not any of its degenerations), we can replace $c_\Lambda$ and $p_\Lambda$ by finite morphisms $c_\Lambda^\Lambda:U_\Lambda^s\rightarrow U_\Lambda$ and $p_\Lambda^\Lambda:U_\Lambda^s\rightarrow V_\Lambda\subset B_\Lambda$, whose degrees are easier to write down (see the proof of Lemma 4.5 in \textit{loc. cit.}). Note that $U_\Lambda^s$ is not an open subset of $D_\Lambda^s$, but is finitely covered by an open subset of $D_\Lambda^s$. Then the degree of $p_\Lambda^\Lambda$ is equal to the number of prong matching equivalence classes and the degree of $c_\Lambda^\Lambda$ is given by the order of the group of ``ghost automorphisms'' and graph automorphisms of $\Lambda$ (see \cite{cmz22} for more details). In particular, if $\Lambda$ has a unique prong-matching equivalence class, then $\text{deg}\,p_\Lambda^\Lambda=1$ and we obtain a finite morphism $V_\Lambda\rightarrow U_\Lambda$.
\end{remark}

Now, assume $\Lambda$ is an enhanced level graph with no horizontal edges. Then the undegeneration of $\Lambda$ that keeps the $i^{\text{th}}$ level passage, that is, the passage between levels $(-i+1)$ and $-i$, and collapses the other level passages, yields a two-level graph $\Gamma_i$ (we will call such two-level graphs divisorial because the corresponding closed boundary stratum is a divisor in $\overline{B}$). Then the (closed) boundary stratum $D_\Lambda$ is contained in the intersection of the divisors $D_{\Gamma_i}$; in fact $D_\Lambda$ is a union of some components of $\cap_{i=1}^LD_{\Gamma_i}$. This collection of divisorial enhanced level graphs $[\Gamma_1,\ldots,\Gamma_L]$ in this exact order is referred to as the \textit{profile} of the boundary stratum in \cite{cmz22}. However, it should be noted that there might be a different level graph $\Lambda'$, also with $L$ levels, such that $D_{\Lambda'}$ is also contained in the intersection $\cap_{i=1}^LD_{\Gamma_i}$. Nonetheless, a consequence of \cite[Prop. 5.1]{cmz22} is that the boundary stratum $D_{\Lambda'}$ also will have the same profile as $D_\Lambda$, that is, the undegeneration of $\Lambda'$ that keeps the $i^{\text{th}}$ level passage and collapses the rest yields exactly $D_{\Gamma_i}$:
\begin{proposition}[\cite{cmz22}, Proposition 5.1]\label{boundary_intersection}
  If $D_{\Gamma_1},\ldots,D_{\Gamma_L}$ are vertical boundary divisors on $\overline{B}$ such that $\cap_{i=1}^LD_{\Gamma_i}$ is non-empty then there is a unique ordering $\sigma\in \mathrm{Sym}_L$ of the set $I\coloneqq\{1,\ldots,L\}$ such that \[D_{\sigma(I)}=\bigcap_{j=1}^LD_{\Gamma_{i_j}},\]where $D_{\sigma(I)}$ is the union of boundary strata in $\overline{B}$ with the profile $[\Gamma_{\sigma(1)},\ldots,\Gamma_{\sigma(L)}]$. Moreover, if $\Gamma_i=\Gamma_j$ for $i\neq j$ then there is no codimension $L$ boundary stratum with profile $[\Gamma_1,\ldots,\Gamma_L]$.
\end{proposition}
\begin{remark}\label{rmk_disjoint}
  If $D_1$ and $D_2$ are two different irreducible components of a vertical boundary divisor $D_\Gamma$, then the preceding proposition implies $D_1\cap D_2$ is empty. Indeed, otherwise $D_1\cap D_2$ would give us a codimension 2 boundary stratum with profile $[\Gamma_1,\Gamma_2]$ for $\Gamma_1=\Gamma_2=\Gamma$. This, in turn, implies that if $D_1$ and $D_2$ are two distinct irreducible components of an arbitrary (not necessarily divisorial) boundary stratum $D_\Lambda$, where $\Lambda$ has at least two levels, then $D_1\cap D_2$ is empty.
\end{remark}

\subsection{Overview of Chow-K\"unneth generation Property}\label{ckgp_overview}
In this section, we aim to summarize the basic results about Chow-K\"unneth generation Property (CKgP) that play an important role in our discussion later. For details and proofs, we refer to \cite[sec. 3]{cl22} and \cite[sec. 4]{clp23}, as well as references therein.
\begin{definition}
  We say a space (scheme or algebraic stack) $X$ has the Chow-K\"unneth generation Property (abbreviated CKgP) if for all spaces $Y$ (of finite type and admitting a stratification by global quotient stacks), the exterior product map on the rational Chow rings:
  \[CH_*(X)_{\mathbb{Q}}\otimes CH_*(Y)_{\mathbb{Q}}\rightarrow CH_*(X\times Y)_{\mathbb{Q}}\]is surjective (where $CH_*(X)_{\mathbb{Q}}\coloneqq CH_*(X)\otimes\mathbb{Q}$).
\end{definition}
Now, we enumerate some important features of CKgP:
\begin{enumerate}
\item\label{p1} If $X$ has the CKgP and $U\subset X$ is open then $U$ also has the CKgP. \cite[Prop. 4.2(1)]{clp23}
\item\label{p2} If $X\rightarrow Y$ is a proper and surjective morphism of DM stacks and $X$ has the CKgP then so does $Y$. \cite[Prop. 4.2(2)]{clp23}
\item\label{p3} If $X$ admits a finite stratification $X=\coprod_{S\in \Sigma}S$ and each $S$ has the CKgP then so does $X$. \cite[Prop. 4.2(3)]{clp23}
\item\label{p4} If $X_1$ and $X_2$ are spaces with the CKgP then $X_1\times X_2$ also has the CKgP. \cite[Lemma 3.2]{cl22}
\item\label{p5} Suppose $X$ is a stack that admits a coarse moduli space $X\rightarrow M$. Then $X$ has the CKgP if and only if $M$ does. \cite[Lemma 3.9]{cl22}
\item\label{p6} If $X$ is a smooth, proper Deligne-Mumford stack that has the CKgP, then the cycle class map $\text{cl}:CH^*(X)\otimes \mathbb{Q}\rightarrow H^*(X,\mathbb{Q})$ is a ring isomorphism. In particular, $X$ has no odd rational cohomology. \cite[Lemma 3.11]{cl22}
\end{enumerate}

\section{Irreducible components of boundary strata}\label{irred_bdry}
The aim of this section is to prove the irreducibility of any non-empty intersection of boundary divisors of the moduli space $\overline{B}$ in genus zero. For the rest of the paper, $\pi$ will denote the birational morphism $\overline{B}\rightarrow \overline{M}_{0,n}$, $D_\Gamma$ will denote the closed boundary stratum given by the level graph $\Gamma$ and $\delta_\Gamma$ its image in $\overline{M}_{0,n}$.
\begin{lemma}\label{unique_graph}
  Suppose $D_{\Gamma_1},\ldots,D_{\Gamma_r}$ are boundary divisors in the moduli space of multiscale differentials $\overline{B}$ in genus 0 with non-empty intersection. Then, there is a unique codimension $r$ level graph compatible with $D_{\Gamma_1}\cap\ldots\cap D_{\Gamma_r}$.
\end{lemma}
\begin{proof}
  We proceed by induction on $r$. First, assume none of the graphs $\Gamma_i$ have a horizontal edge, and $\Lambda_1$ and $\Lambda_2$ are codimension $r$ level graphs compatible with some element in $D_{\Gamma_1}\cap\ldots\cap D_{\Gamma_r}$. \Cref{boundary_intersection} implies that there is the unique ordering (which is the same for both $\Lambda_1$ and $\Lambda_2$) of the graphs $\Gamma_1,\ldots,\Gamma_r$ so that $\Gamma_i$ is obtained from both $\Lambda_1$ and $\Lambda_2$ via the undegeneration keeping the $-i+1$ to $-i$ level passage and collapsing the others. We assume that $D_{\Gamma_2}\cap\ldots\cap D_{\Gamma_r}$ is given by a unique codimension $r-1$ level graph $\Delta$. So, both $\Lambda_1$ and $\Lambda_2$ are given by the top level degeneration of $\Delta$, so that the vertices in level -2 and below of both $\Lambda_1$ and $\Lambda_2$ are identical (to those at $-1$ and below of $\Delta$) -- that is, the corresponding vertices have identical set of half edges. Additionally, for each level passage below level $-1$, there is one-to-one correspondence between the edges in $\Lambda_1$ and those in $\Lambda_2$, that matches the enhancements as well. On the other hand, the undegenerations of both $\Lambda_1$ and $\Lambda_2$ keeping the top level passage (and collapsing the others) yield $D_{\Gamma_1}$. This forces both of them to have the same vertices at the top level as well; also the edges appearing in the top level passages of both $\Lambda_1$ and $\Lambda_2$ are in one-to-one correspondence. Since the level graphs are identical at level -2 and lower, we can contract all the lower level passages and assume that $\Lambda_1$ and $\Lambda_2$ have exactly three levels, with identical sets of vertices at levels 0 and $-2$.
  
  Now, we consider the vertices appearing in level $-1$ and prove that there is a one-to-one correspondence between the vertices of $\Lambda_1$ and $\Lambda_2$ with identical sets of half-edges. Consider a vertex $u_1$ in level $-1$ of $\Lambda_1$ that contains a marked point $p_1$. Suppose the vertex in level $-1$ of $\Lambda_2$ with the marked point $p_1$ is $u_1'$. We will now prove that both $u_1$ and $u_1'$ have identical set of marked points. So suppose the contrary, that is, there is a marked point $p_2$ on $u_1$ that is contained on a different vertex $u_2'$ in $\Lambda_2$. Since contracting the top level passage results in identical level graph $\Delta$ for both $\Lambda_i$'s, both $p_1$ and $p_2$ will be on the same top level vertex of $\Delta$. Consequently, both $u_1'$ and $u_2'$ must be connected via a sequence of vertical edges whose endpoints are contained in level 0 or $-1$, that is, we have a path $u_1'\rightarrow \alpha_1\rightarrow \ldots\rightarrow \alpha_k\rightarrow u_2'$ in $\Lambda_2$ consisting of vertical edges, where $\alpha_i$ are vertices in level 0 or $-1$. Similarly, contracting the bottom level passage next, we obtain a path (composed of vertical edges) $u_1'\rightarrow \beta_1\rightarrow\ldots\rightarrow \beta_j\rightarrow u_2'$ where $\beta_i$ are vertices in level $-1$ or $-2$. Combining these two paths, we obtain a non-trivial loop in $\Lambda_2$, which cannot exist on a tree. Thus, $u_1$ and $u_1'$ have identical set of marked points if they have one in common.
  
  Next, we look at a vertex $u$ in $\Lambda_1$ that has no marked point; our aim will be to find a vertex $u'$ in $\Lambda_2$ that has identical set of half-edges as $u$, that is $u'$ also has no marked point, and is connected to the identical set of vertices in level 0 and level $-2$ as $u$. Note that there must be at least three edges attached to $u$, with at least one going up and one going down. Suppose $u$ is connected to vertices $v_1,\ldots, v_r$ in the top level and $w_1,\ldots,w_s$ in the bottom level. Denote by $v_1',\ldots,v_r'$ the vertices in the top level of $\Lambda_2$ with the same set of half edges as $v_1,\ldots,v_r$ and by $w_1',\ldots,w_s'$ the analogous set of vertices in the bottom level of $\Lambda_2$. As before, contracting the bottom level passage must yield the same level graph for both $\Lambda_i$'s. Since such undegeneration coalesces $u,w_1,\ldots,w_s$ to a single vertex, it follows that $w_1',\ldots,w_s'$ also get coalesced into a single vertex after such a degeneration. But this can only happen if any two of $w_i'$ are connected via path consisting of vertical edges whose endpoints lie on level $-1$ or $-2$. Suppose such a subtree consists of two (or more) vertices $u_1', u_2'$ of level $-1$. Analogously, it follows that all of $v_i'$ are also connected via a subtree consisting of vertical edges whose endpoints lie on level 0 and $-1$. Then $u_1'$ and $u_2'$ also should be a part of such a subtree, so we see that $u_1'$ and $u_2'$ are connected by two distinct paths, which is impossible in a tree. Thus, it follows that the vertices $v_1',\ldots,v_r'$ and $w_1',\ldots,w_s'$ are connected to a unique vertex $u'$ in level $-1$, and this vertex has identical set of half-edges as $u$ in $\Lambda_1$.
  
  Combining the argument from the previous two paragraph, we conclude that there is one-to-one correspondence between the vertices in level $-1$ of $\Lambda_1$ and $\Lambda_2$ with identical set of half edges.
  
  Next, we assume $D_{\Gamma_1}$ is a horizontal boundary divisor. Such a horizontal edge is characterized by the partition of marked points into the two ends. By induction hypothesis, we assume $D_{\Gamma_2}\cap\ldots\cap D_{\Gamma_r}$ is given by a unique codimension $r-1$ graph $\Delta$. Then the dual graph of $\Lambda_i$ has one more edge than that of $\Delta$ and that edge is created by a degeneration of a particular vertex $v$ of the latter. That is, both $\Lambda_1$ and $\Lambda_2$ are obtained from $\Delta$ by degenerating the vertex $v$ in $\Delta$ by introducing a horizontal edge that produces a fixed partition of the marked points on the two sides, so they have to be the same.
\end{proof}
Now we move to the proof of irreducibility of the boundary strata. First we deal with the level graphs with only horizontal edges.
\begin{claim}\label{irreducibility_divisor}
  Suppose $\Gamma$ is a level graph with  only horizontal edges. Then the boundary stratum $D_\Gamma\subset \overline{B}$ is irreducible.
\end{claim}
\begin{proof}
  Suppose the level graph $\Gamma$ has $r$ horizontal edges, so $D_\Gamma$ has codimension $r$ in $\overline{B}$. Then the image $\delta_\Gamma$ of $D_\Gamma$ in $\overline{M}_{0,n}$ also has codimension $r$, so the fiber of $\pi:\overline{B}\rightarrow \overline{M}_{0,n}$ over a generic point in $\delta_\Gamma$ is finite. But $\overline{M}_{0,n}$ is normal, so by Zariski's main theorem, the fiber of $\pi$ over a generic point in $\delta_\Gamma$ is connected (so a singleton), and thus $D_\Gamma$ is irreducible.
\end{proof}
For a more general level graph $\Gamma$, a priori the presence of multiple prong-matching equivalence classes for a twisted differential, that cannot be connected by the monodromy action of the level rotation torus, can result in reducibility of the boundary stratum $D_\Gamma$. However, the prong-matching equivalence classes are better behaved in genus 0, as indicated by the following lemma, which will play a crucial role in our proof of irreducibility of the boundary strata:
\begin{lemma}\label{prongs_unique}
  Suppose $\Gamma$ is an enhanced level graph in genus 0 with only vertical edges (or more generally, an enhanced level graph that is a tree and has only vertical edges) such that each level contains exactly one vertex. Then the number of prong matching equivalence classes for any twisted differential compatible with the level graph $\Gamma$ is exactly one.
\end{lemma}
Utmost thanks to Myeongjae Lee for providing the idea of the proof of this lemma.
\begin{proof}
  The proof is based on induction on the number of levels on $\Gamma$ (or equivalently, the number of edges on $\Gamma$). The base case where the graph $\Gamma$ has exactly one edge is clear. Now, consider a level graph $\Gamma$ with $L$ edges (so there are $L+1$ vertices and $L$ levels below the top one). Let $v$ be a leaf of $\Gamma$ (since $\Gamma$ is a tree, such a vertex exists) and assume it is connected by (unique) edge $e$ to the unique vertex at level $j$ (without loss of generality, assume $i>j$). Erasing the edge $e$ and replacing it by the corresponding half leg on the vertex at level $j$, we obtain a new enhanced level graph $\Lambda$ with $L$ vertices. Then, by the induction hypothesis, the level rotation group $R_\Lambda\cong\mathbb{Z}^{L-1}$ acts on the set of prong matchings on a twisted differential compatible with $\Lambda$ transitively. This level rotation group $R_\Lambda$ is a subgroup of the level rotation group $R_\Gamma\cong \mathbb{Z}^L$ as a direct summand, $R_\Gamma=R_\Lambda\times \mathbb{Z}$, where the last factor of $\mathbb{Z}$ acts on the prong matchings of $\Gamma$ by rotating the prongs in the level passage from $i$ to $i-1$ (but fixing the prongs in other level passages), whereas the factor $R_\Lambda$ acts by rotating the prongs in the other level passages.
  
  So, using the transitivity of the aciton of $R_\Lambda$ on the prongs matchings on $\Lambda$, we  identify all the prong matchings on the edges other than $e$, then use the remaining factor of $\mathbb{Z}$ acting on the level passage from $i$ to $i-1$ in order to rotate the prongs on edge $e$ so as to identify any two prong matchings on $e$ to each other. However this latter action also rotates the prongs on all the other edges that cross the level passage $i$ to $i-1$. But notice that since $v$ is the unique vertex in level $i$, any edge in $\Gamma$ that crosses the level passage $i$ to $i-1$ also crosses the level passage $i+1$ to $i$. So, if $p$ is the integer by which we rotated the level passage $i$ to $i-1$, we readjust the prongs in the level passage $i+1$ to $i$ via rotation by $-p$, so that the net result prong rotation of all edges (other than $e$) that crosses the level passage $i$ to $i-1$ is 0 (in particular, this ensures the transitivity of action of $R_\Lambda$ on the edges other than $e$ remains undisturbed). This completes the proof.
\end{proof}
We remark that the same proof applies even when there are some horizontal edges, as long as each level by itself is connected. With this, we can prove that the boundary stratum $D_\Gamma$ for such a graph $\Gamma$ is irreducible:
\begin{proposition}
  Suppose $\Gamma$ is an enhanced level graph in genus 0 such that each level by itself is connected, i.e., the vertices within each level are connected by horizontal edges. Then the boundary stratum $D_\Gamma$ is irreducible.
\end{proposition}
\begin{proof}
  Since the number of prong-matching equivalence classes is one, using \cref{degree_product_stratum}, we see that the open substratum $U_\Gamma\subset D_\Gamma$ (consisting of multiscale differentials compatible with $\Gamma$, but not with any degeneration of $\Gamma$) is finitely covered by $V_\Gamma\subset B_\Gamma$ (we are using the same notation as in the remark). Note that $V_\Gamma$ is simply a product of level-wise open strata: $V_\Gamma=\prod_{i=0}^LV_\Gamma^{[i]}$, where $V_\Gamma^{[i]}\subset B_\Gamma^{[i]}$ is the open substratum consisting of non-degenerate twisted differentials. Thus, to prove that $U_\Gamma$ (and thus $D_\Gamma$) is irreducible, it is enough to prove that each $B_\Gamma^{[i]}$ is irreducible. If $i^{\text{th}}$ level consists of a single vertex then since Global Resiue Condition does not appear in genus 0, $B_\Gamma^{[i]}$ is simply $\mathbb{P}\Xi\overline{\mathcal{M}}_{0,n_i}(\mu^{[i]})$ for appropriate $n_i$ and $\mu^{[i]}$, so is irreducible. On the other hand, if the $i^{\text{th}}$ level has horizontal edges, by smoothing the horizontal edges, we can embed $B_\Gamma^{[i]}$ into some moduli space $\mathbb{P}\Xi\overline{\mathcal{M}}_{0,n_i}(\mu^{[i]})$, for appropriate $n_i$ and $\mu^{[i]}$, as a boundary stratum defined by a level graph with only horizontal edges. Then by \cref{irreducibility_divisor}, such a boundary stratum has to be irreducible.
\end{proof}

Using the preceding proposition as the base case, we will now extend the irreducibility to the arbitrary boundary stratum $D_\Gamma$:
\begin{proposition}\label{irreducible_boundary}
  Suppose $\Gamma$ is an (enhanced) level graph in genus 0, then the boundary stratum $D_\Gamma$ parametrizing the multiscale differentials  compatible with $\Gamma$ is irreducible.
\end{proposition}
This, in particular, implies that any non-empty intersection $D_{\Gamma_1}\cap\ldots\cap D_{\Gamma_i}$ of boundary divisors is irreducible.
\begin{proof}
  Since purely horizontal level graphs are already covered by \cref{irreducibility_divisor}, we assume that $\Gamma$ has at least two levels. Then we degenerate $\Gamma$ to $\Delta$ (without changing the underlying dual graph) so that each level of $\Delta$ is connected. That is, this degeneration only involves moving the various connected components of a particular level up or down. Then, as proven in the preceding proposition, $D_\Delta$ is irreducible. However, any irreducible component of $D_\Gamma$ must admit a degeneration to differentials compatible with the level graph $\Delta$, so $D_\Delta$ is contained in the intersection of the irreducible components of $D_\Gamma$. But since $\Gamma$ contains a vertical edge, by \cref{boundary_intersection}, the components of $D_\Gamma$ have to be pairwise disjoint (see Remark \ref{rmk_disjoint}), and thus $D_\Gamma$ can only have one component.
\end{proof}
\begin{remark}
In higher genus, there trivially exist vertical divisors that are reducible. As an example, consider $\mu=(1,3,-4)$. Then by \cite[Prop. 3.3]{cc14}, the moduli space $\overline{B}=\mathbb{P}\Xi\overline{\mathcal{M}}_{1,3}(\mu)$ is irreducible. Now, consider the enhanced level graph $\Gamma$
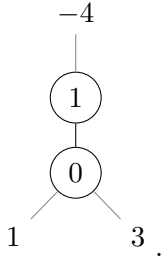
\begin{figure}[H]
  \centering
\begin{tikzpicture}[main/.style = {draw, circle}] 
  \node[main, minimum size=2mm,pin=90:$-4$] (1) {$1$};
  \node[main, minimum size=2mm,pin=-135:$1$, pin=-45:$3$] (2) [below of=1]{$0$};
  \draw (1) to [out=-90,in=90] (2);
\end{tikzpicture}.
\end{figure}
Then $D_\Gamma$ has two irreducible components. Indeed, the top level stratum of $D_\Gamma$ is isomorphic to the stratum $\mathbb{P}\Xi\overline{\mathcal{M}}_{1,2}(4,-4)$, which, by \cite[Prop. 3.2]{cc14}, has two irreducible components.
\end{remark}
\section{Computation of rational cohomology}\label{rational_cohomology}
The main objective of this section is to prove \cref{main-thm}. For this, we will construct a stratification of the moduli space $\overline{B}$ into locally closed subspaces with CKgP. Because of the property \ref{p5} in the section \ref{ckgp_overview}, whether we produce a stratification for the stack or its coarse moduli space is immaterial, however for consistency, we will stick with the coarse space $\overline{B}$. For each boundary stratum $D_\Gamma$ (associated to the level graph $\Gamma$), we denote by $D_\Gamma^o\subset D_\Gamma$ the open substratum consisting of differentials compatible with $\Gamma$ but none of its degenerations (that is, both the degenerations of $\Gamma$ that change the underlying dual graph as well as the degenerations of $\Gamma$ that only change the level structure, but keep the dual graph fixed, are excluded from $D_\Gamma^o$). Similarly, $\delta_\Gamma$ will denote the image of $D_\Gamma$ in $\overline{M}_{0,n}$ and $\delta_\Gamma^o\subset \delta_\Gamma$ the image of $D_\Gamma^o$ in $\overline{M}_{0,n}$.
\begin{proposition}
  For each open boundary stratum $D_\Gamma^o$, there is a finite and surjective morphism $\delta_\Gamma^o\times (\mathbb{C}^*)^d\rightarrow D_\Gamma^o$, where $d$ is the dimension of a generic fiber of the morphism $\pi|_{D_\Gamma^o}$.
\end{proposition}
\begin{proof}
  A point of $D_\Gamma^o$ is represented by the data of a nodal stable curve $C$ of genus 0 and marked points $p_1,\dots,p_n$, along with a collection of differentials, one for each component of $C$, with zeros and poles of prescribed order dictated by $\mu$ and compatible with the level graph $\Gamma$ (with opposite residues on the horizontal edges), and compatible prong matchings between the edges connected by a vertical edge. Each component of $C$ is a smooth rational curve, and the data of half edges (along with the order of zeros and poles) on the vertex corresponding to that component will define an unprojectivized moduli space of abelian differentials over a smooth genus 0 curve. Any such abelian differential can be written as $f(z)dz$, where $z$ is a global coordinate on $\mathbb{P}^1$ (which can be chosen to vary holomorphically over $\delta_\Gamma^o$) and $f(z)$ is a rational function with zeros and poles prescribed by the half edges and their enhancements. Suppose the $p^{th}$ level of an element in $D_\Gamma^o$ has connected components $C_1,\ldots,C_m$; each $C_i$ is a union of rational components connected by horizontal edges. A compatible twisted differential on $C_i$ will then be a collection of differentials $\omega_i=\{\omega_{i,j}\}$ of the form $f_{i,j}(z)dz$ on each irreducible component $R_{i,j}$ of $C_i$, where $f_{i,j}(z)$ on each irreducible component is chosen so that each horizontal edge will have a pole with opposite residues on the two vertices connected by the edge. Then all compatible twisted differentials on $C_i$ will simply be scalar multiples of $\omega_i$. Note that up to a scalar factor (say $K$), the rational function $f_{i,j}(z)$ is a quotient of a product of linear factors $(z-a)$ (where zeros appear in the numerator and poles in the denominator, and the factors are repeated as indicated by the order, of course). As we vary the curve $C_i$ holomorphically in the moduli, the residues of the simple poles at the horizontal edges also vary holomorphically. Consequently, the factor $K$ required to make sure the two sides of a horizontal edge have opposite residues, will also vary holomorphically. Doing this for each level, we obtain a section $\sigma$ of the morphism $D_\Gamma^o\rightarrow\delta_\Gamma^o$ (for prong-matching, we fix any choice of orientation-reversing bijection between the outgoing prongs on the higher vertex and the incoming prongs on the lower vertex for each vertical edge).
  
  Recall that there is an action of a torus $(\mathbb{C}^*)^{V(\Gamma)}$ on the space of multiscale differentials compatible with $\Gamma$ that acts by rescaling the differentials on each vertex of $\Gamma$ individually. In general, this does not respect the residue condition on the horizontal edges, so there is a subtorus $T_{RC}\subset (\mathbb{C}^*)^{V(\Gamma)}$ that preserves the differentials in a fiber of $D_\Gamma^o\rightarrow\delta_\Gamma^o$. But this torus does not bring the prong matchings into consideration -- a loop that is non-trivial in the fundamental group of $T_{RC}$ might return to the same differential but with a different prong matching belonging to a different equivalence class. So, one might need to take a finite etale cover of $T_{RC}$ to make sure the torus acts naturally; however we will continue to denote this torus also by $T_{RC}$. The actions of various such tori and how they interact with the prong-matchings is discussed in detail in \cite[Sec. 5]{bcggm}. Note that the torus $T_{RC}$ acts transitively on the fibers, so using the section $\sigma$ mentioned above, we obtain a surjective morphism $\eta:\delta_\Gamma^o\times T_{RC}\rightarrow D_\Gamma^o$. The fiber of this morphism will be a torus isogeneous to $(\mathbb{C})^{L(\Gamma)}$ acting by rescaling the differentials level by level. Thus $\eta$ factors through a finite surjective morphism $T\times \delta_\Gamma^o\rightarrow D_\Gamma^o$ for a torus $T\cong(\mathbb{C}^*)^d$ of appropriate dimension. In fact this torus is analogous to the one constructed in \cite{ccm22} to verify that each component of a fiber of $\overline{B}\rightarrow \overline{M}_{0,n}$ is covered by a toric variety.
  \end{proof}
Since $\delta_\Gamma^o$ is a product of $M_{0,i}$, which has the CKgP, and a torus, being an open subset of the affine space, also has the CKgP (see Property \ref{p1} of CKgP in section \ref{ckgp_overview}), it follows from Properties \ref{p2} and \ref{p4} of CKgP that the open boundary stratum $D_\Gamma^o\subset\overline{B}$ also has the CKgP. This implies $\overline{B}$ admits a stratification by spaces with CKgP, and thus by Property \ref{p3}, $\overline{B}$ itself has the CKgP. Since $\overline{\mathcal{B}}$ is a smooth and proper Deligne Mumford stack, using Property \ref{p6} of CKgP, yields:
\begin{proposition}
  The rational Chow and cohomology rings of $\overline{B}$ are isomorphic: $CH^*(\overline{B})\otimes\mathbb{Q}\cong H^*(\overline{B},\mathbb{Q})$. In particular, $\overline{B}$ has no odd rational cohomology.
\end{proposition}
Next, we will prove that the rational Chow ring is isomorphic to the tautological ring $R^*(\overline{B})$. Recall from \cite[Sec. 8]{cmz22} that the tautological rings of strata are defined as the smallest set of $\mathbb{Q}$-algebras $R^*(\mathbb{P}\Xi\overline{\mathcal{M}}_{g,n}(\mu))\subset CH^*(\mathbb{P}\Xi\overline{\mathcal{M}}_{g,n}(\mu))\otimes\mathbb{Q}$ which:
\begin{itemize}
\item contain the $\psi$-classes for each marked point,
\item is closed under the pushforward of the map forgetting a marked point of order zero, and
\item is closed under the maps $(c_\Gamma)_*(p^{[i]}_\Gamma)^*$ that we introduced in \cref{product_stratum}, for each enhanced level graph $\Gamma$.
\end{itemize}
We will define the tautological ring of any (closed) boundary stratum $D_\Lambda$ analogously (in the third bullet point, we only consider the enhanced level graphs that are degenerations of $\Lambda$, of course). Similarly, the tautological ring of an open subset (either of $\mathbb{P}\Xi\overline{\mathcal{M}}_{g,n}(\mu)$ or a boundary stratum) will be defined as the restriction of the tautological ring from $\mathbb{P}\Xi\overline{\mathcal{M}}_{g,n}(\mu)$ or the boundary sratum, respectively.

Now, we specialize to genus $g=0$.
\begin{proposition}
  The rational Chow and tautological rings of $\overline{B}$ are isomorphic: $CH^*(\overline{B})\otimes\mathbb{Q}\cong R^*(\overline{B})$.
\end{proposition}
\begin{proof}
  The idea is similar to \cite[Lemma 4.1]{cl22} (\textit{Filling Criteria: version 1}). As observed above, all boundary strata of $\overline{B}$ have CKgP and provide a stratification of $\overline{B}$ (thereby also proving $\overline{B}$ has CKgP). We will next prove that each boundary stratum also satisfies $CH^*=R^*$.
  
  First note that the interior $B\coloneqq\overline{B}\setminus\partial\overline{B}$ is equal to $M_{0,n}$ which is a Zariski open subset of the affine space $\mathbb{C}^{n-3}$. So the Chow ring of $M_{0,n}$ (and thus of $B$) is isomorphic to $\mathbb{Q}$ generated by its fundamental class; in particular it is tautological.
  
  Similarly, the interior $D_\Gamma^o$ of a boundary stratum $D_\Gamma$ also satisfies $CH^*=R^*$ since, as discussed above, $D_\Gamma^o$ is a finite image of $(\mathbb{C}^*)^r\times\delta_\Gamma^o$ which also has Chow ring isomorphic to $\mathbb{Q}$ (note that by the CKgP, the morphism $CH^*((\mathbb{C}^*)^r)\otimes CH^*(\delta_\Gamma^o)\rightarrow CH^*((\mathbb{C}^*)^r\times\delta_\Gamma^o)$ is surjective, and both $(\mathbb{C}^*)^r$ and $\delta_\Gamma^o$ have Chow ring isomorphic to $\mathbb{Q}$, being an open subset of the affine space).
  
  Now, we prove that for each $\Gamma$, the boundary stratum $D_\Gamma$ has its rational Chow ring isomorphic to its tautological ring, by induction on the dimension of $D_\Gamma$. The base case ($\text{dim}=0$) is vacuous. Now, assume $D_\Gamma$ is an $r$-dimensional boundary stratum represented by the level graph $\Gamma$. Then $D_\Gamma$ admits a stratification into open boundary substrata (as well as its interior $D_\Gamma^o$) such that each has CKgP and satisfies $CH^*=R^*$. Further, by induction hypothesis, each closed boundary substratum of $D_\Gamma$ also has CKgP and satisfies $CH^*=R^*$. By the localization sequence for the Chow groups (\cite[Prop. 1.8]{fulton}), there is a surjective morphism $CH_k(D_\Gamma\setminus D_\Gamma^o)\oplus CH_k(D_\Gamma^o)\rightarrow CH_k(D_\Gamma)$. Since both direct summands on left side are generated by tautological classes, it follows that the Chow ring of $D_\Gamma$ is also generated by the tautological classes, thereby completing the proof.
\end{proof}
With the application of \cref{irreducible_boundary}, we obtain:
\begin{proposition}
  The rational Chow ring (and thus, also rational cohomology ring) of $\overline{B}$ is generated by the boundary divisors.
\end{proposition}
\begin{proof}
  Since the Chow ring of $\overline{B}$ is tautological, every rational cycle class in $\overline{B}$ is a pushforward of a cycle class in the boundary. Inducting on the dimension, we will prove that every irreducible boundary stratum $D_\Gamma$ has its rational Chow ring generated by its boundary divisors, where by boundary divisor of $D_\Gamma$, we refer to the divisors in $D_\Gamma$ obtained as its proper intersection with a boundary divisor in $\overline{B}$. The base case ($\text{dim}=0$) is clear. Now, for the induction step, take a boundary stratum $D_\Gamma$ and consider an irreducible cycle $Z\subset D_\Gamma$. Since the interior of $D_\Gamma$ has the Chow ring isomorphic to $\mathbb{Q}$, we can assume $Z$ is contained in the boundary; in particular, take $Z\subset D_\Lambda$ for a one level degeneration $\Lambda$ of $\Gamma$. So, by induction hypothesis, $D_\Lambda$ has its Chow ring generated by boundary divisors. But the boundary divisors in $D_\Lambda$ are boundary strata in $D_\Gamma$ represented by appropriate level graphs obtained as codimension one degeneration of $\Lambda$, so by Lemma \ref{unique_graph} and Proposition \ref{irreducible_boundary}, are restrictions of boundary divisors of $D_\Gamma$ into $D_\Lambda$, which completes the proof.
\end{proof}
This also concludes the proof of \cref{main-thm}.
\begin{remark}
  The applicability of CKgP is very special to genus 0. In fact, already in genus 1 and $n=2$, the CKgP fails for the moduli space of differentials $\overline{\mathcal{B}}\coloneqq \mathbb{P}\Xi\overline{\mathcal{M}}_{1,2}(p,-p)$ if $p$ is large enough. Indeed, assuming $p$ is prime, this moduli space is the classical modular curve $X_1(p)$ whose genus grows to infinity as $p\rightarrow\infty$ -- see \cite{modular} (or \cite{tahar17} for flat-geometric proof). Then by the property \ref{p6} of CKgP in \cref{ckgp_overview}, it cannot have the CKgP.\\
  Actually, even when the coarse moduli space $\overline{B}$ of $\mathbb{P}\Xi\overline{\mathcal{M}}_{1,n}(\mu)$ is birational to $\overline{M}_{1,n-1}$ and $n-1\leq 10$ (the range where $\overline{M}_{1,n-1}$ has the CKgP, see \cite[Section 5]{cl22}), we cannot expect the techniques we used in this paper to be applicable. Indeed, if $\mu=(1,m,-m-1)$, then $\overline{B}$ admits a boundary divisor associated to the enhanced level graph
  
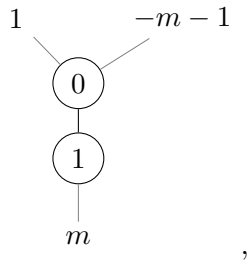
\begin{figure}[H]
  \centering
\begin{tikzpicture}[main/.style = {draw, circle}] 
  \node[main, minimum size=2mm,pin=-90:$m$] (1) {$1$};
  \node[main, minimum size=2mm,pin=135:$1$, pin=45:$-m-1$] (2) [above of=1]{$0$};
  \draw (1) to [out=90,in=-90] (2);
\end{tikzpicture},
\end{figure}
\noindent which has a component isomorphic to the modular curve $X_1(m)$. So, the stratification of $\overline{B}$ into boundary strata alone is not sufficient to show that $\overline{B}$ has the CKgP when $m$ becomes large enough.
\end{remark}

\subsection{Relations in the cohomology}\label{relations_cohomology}
It is proven in \cite{km94} that the space of relations between the codimension $i$ boundary strata in the cohomology group $H^{2i}(\overline{M}_{0,n})$ is \textit{additively} generated by the pullbacks of the WDVV relations from $\overline{M}_{0,4}$. Our aim in this subsection is to prove the analogous statement on the space of relations in the cohomology ring of the moduli space of differentials $\overline{B}$ in genus 0. Our motivation for the argument instead comes from the Mixed Hodge Theoretic argument used in \cite{petersen12} (which, in turn, is based on the ideas of \cite{getzler95}). In particular, we will see that while the space of relations is not necessarily $\mathbb{Q}$-linearly generated by the pullbacks of the WDVV relations from $\overline{M}_{0,n}$, the other generators can be recovered geometrically from the toric variety structure on the fibers of the birational morphism $\pi:\overline{B}\rightarrow\overline{M}_{0,n}$.

The idea is to consider the spectral sequence of a filtration. Indeed, we consider the sequence of subvarieties $T_0\subset T_1\subset\ldots \subset T_{n-3}=\overline{B}$, where $T_p$ is the union of all the boundary strata $D_\Gamma\subset \overline{B}$ of dimension at most $p$. Then we have a spectral sequence in cohomology with compact support (with $\mathbb{Q}$-coefficients, as usual)
\begin{equation}\label{spectral_sequence}
  E_1^{p,q}=H^{p+q}_c(T_p\setminus T_{p-1})\implies H^{p+q}_c(\overline{B}),
\end{equation}
  whose differentials are compatible with the natural Mixed Hodge Structure (MHS). Note that $T_p\setminus T_{p-1}=\sqcup_\Gamma D_\Gamma^o$ is the union of the open boundary strata of dimension exactly $p$ (which are all pairwise disjoint). Of course $\overline{B}$ is compact with only orbifold singularities, so $H_c^{p+q}(\overline{B})=H^{p+q}(\overline{B})$ has pure Hodge structure of weight $p+q$. Consequently, only $gr_{p+q}^WE_1^{p,q}$ survives to the $E_\infty$ page of the spectral sequence, that is, we have to only understand $gr_i^W H^i_c(D_\Gamma^o)$. Poincar\'e Duality gives the perfect pairing
\[H^{2p-i}(D_\Gamma^o)\otimes H^i_c(D_\Gamma^o)\rightarrow H_c^{2p}(D_\Gamma^o)\cong\mathbb{Q},\]
that is compatible with the natural mixed Hodge structures and weights, where the right side has the pure Hodge structure of weight $2p$ given by the fundamental class. Thus, determining $gr_i^WH^i_c(D_\Gamma^o)$ is equivalent to determining $gr_{2p-i}^WH^{2p-i}(D_\Gamma^o)$. We know that the mixed Hodge structure on $k^{\text{th}}$ cohomology group $H^k$ of a complement of a hyperplane arrangement in $\mathbb{C}^N$ is pure of weight $2k$ and of type $(k,k)$ (see \cite{shapiro93}), so using the finite morphism $\delta_\Gamma^o\times (\mathbb{C}^*)^r\rightarrow D_\Gamma^o$, we see that $H^i(D_\Gamma^o)$ also has pure Hodge structure of weight $2i$ and of type $(i,i)$; equivalently, $H^i_c(D_\Gamma^o)$ has pure Hodge structure of weight $2i-2p$. In particular, this implies $gr_i^WH^i(D_\Gamma^o)\neq 0$ if and only if $i=0$; equivalently, $gr_i^WH^i_c(D_\Gamma^o)\neq 0$ if and only if $i=2p$.

Summarizing the above discussion, we have proven the following:
\begin{lemma}
  The spectral sequence (\ref{spectral_sequence}) is a first quadrant spectral sequence on $E_1$-page such that $E_1^{p,q}=0$ for $q\geq p$ and $E_1^{p,q}$ has the pure Hodge structure of weight $2(p+q)-2p=2q$ for $p\geq q$.
\end{lemma}
The pure weights at $E_1^{p,q}$ of the spectral sequence is illustrated as follows:
\begin{figure}[H]
  \centering
\begin{tikzpicture}
  \matrix (m) [matrix of math nodes,
    nodes in empty cells,nodes={minimum width=5ex,
    minimum height=5ex,outer sep=-5pt},
    column sep=1ex,row sep=1ex]{
                &      &     &     &     & \\
          2     &   *  &  *  &  4  &  4  & \\
          1     &   *  &  2  &  2  &  2  & \\
          0     &   0  &  0  &  0  &  0  & \\
    \quad\strut &   0  &  1  &  2  &  3  & \strut \\};
\draw[-stealth] (m-2-4) -- (m-2-5);
\draw[-stealth] (m-3-3) -- (m-3-4);
\draw[-stealth] (m-3-4) -- (m-3-5);
\draw[-stealth] (m-4-2) -- (m-4-3);
\draw[-stealth] (m-4-3) -- (m-4-4);
\draw[-stealth] (m-4-4) -- (m-4-5);
\draw[thick] (m-1-1.east) -- (m-5-1.east) ;
\draw[thick] (m-5-1.north) -- (m-5-6.north) ;
\end{tikzpicture}.
\end{figure}
\noindent (The asterisk indicates the corresponding $E_1^{p,q}$ is trivial and the arrows indicate the differentials $d_1$ on the $E_1$-page.) Consequently, the differentials $d_r:E_r^{p,q}\rightarrow E_r^{p+r,q-r+1}$ are trivial when $r\geq 2$ (because the pure weights of the two sides are different), that is, $E_2^{p,q}=E_\infty^{p,q}$. Moreover, only $E_1^{p,p}$ can survive to $E_\infty$, so we have
\[H_c^{2p}(\overline{B})=E_2^{p,p}=\text{ker}\left(H^{2p}_c(T_p\setminus T_{p-1})\rightarrow H^{2p+1}_c(T_{p+1}\setminus T_p)\right).\]Dualizing, we see that the map $H^0(T_p\setminus T_{p-1})\rightarrow H^{2n-2p}(\overline{B})$, given by the Gysin pushforward, is surjective and that the sequence
\[H^1(T_{p+1}\setminus T_p)\rightarrow H^0(T_p\setminus T_{p-1})\rightarrow H^{2n-2p}(\overline{B})\rightarrow 0\] is exact. That is, the space of relations among the $p$-dimensional boundary strata in cohomology is given by the image of $H^1(T_{p+1}\setminus T_p)$ in $H^0(T_p\setminus T_{p-1})$. Note that we have $H^1(T_{p+1}\setminus T_p) = \bigoplus_{\Gamma\in LG^{p+1}}H^1(D_\Gamma^o)$, where $LG^{p+1}$ is the set of all level graphs $\Gamma$ such that $D_\Gamma$ is $(p+1)$-dimensional, so the direct sum of $H^1(\delta_\Gamma^o)\oplus H^1((\mathbb{C}^*)^{r(\Gamma)})\cong H^1(\delta_\Gamma^o\times (\mathbb{C}^*)^{r(\Gamma)})$, over all the level graphs $\Gamma\in LG^{p+1}$, surjects onto the space of relations (where $r(\Gamma)$ is the dimension of a general fiber of $D_\Gamma^o\rightarrow \delta_\Gamma^o$). As mentioned in \cite{petersen12}, the image of $H^1(\delta_\Gamma^o)$ constitutes the WDVV relations pulled back from $\overline{M}_{0,n}$. On the other hand, since a connected component of a fiber of $\overline{B}\rightarrow \overline{M}_{0,n}$ is finitely covered by a simplicial toric variety (see \cite{ccm22} for details), the image of $H^1((\mathbb{C}^*)^{r(\Gamma)})$ consititute the relations between the $p$-dimensional torus invariant subvarieties in the fiber of $D_\Gamma\rightarrow\delta_\Gamma$. Thus, we have proven the following:
\begin{proposition}
  The vector space of relations between the cohomology classes of $p$-dimensional boundary strata in $H^{2d-2p}(\overline{B})$ (where $d=n-3=\text{dim}\,\overline{B}$) is $\mathbb{Q}$-linearly generated by the WDVV relations pulled back from $\overline{M}_{0,n}$ and the relations between the $p$-dimensional torus invariant subvarieties in the toric fibers of $\overline{B}\rightarrow\overline{M}_{0,n}$.
\end{proposition}
\begin{example}
  Suppose $n=7$ and $\mu=(0^6,-2)$, then $\overline{B}=\mathbb{P}\Xi\overline{M}_{0,7}(\mu)$ is smooth (this is proven in the next section). Then consider the boundary stratum $D_\Gamma$ given by the level graph $\Gamma$:
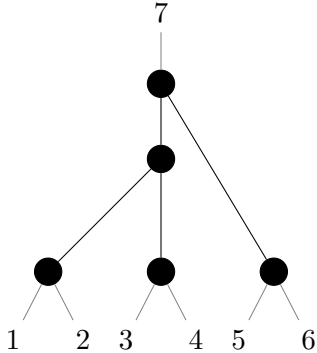
  \begin{figure}[H]
    \centering
  \begin{tikzpicture}[main/.style = {draw, circle, fill=black}] 
  \node[main, minimum size=2mm,pin=90:$7$] (5) at (0,1) {};
  \node[main, minimum size=2mm] (1) at (0,0) {};
  \node[main, minimum size=2mm,pin=250:$5$,pin=290:$6$] (2) at (1.5,-1.5) {};
  \node[main, minimum size=2mm,pin=250:$1$,pin=290:$2$] (3) at (-1.5,-1.5) {};
  \node[main, minimum size=2mm,pin=250:$3$,pin=290:$4$] (4) at (0,-1.5) {};
  \draw (5) to (2);
  \draw (1) to (3);
  \draw (1) to (4);
  \draw (1) to (5);
\end{tikzpicture}
\end{figure}
\noindent (where the legs are labelled by their indices as a marked point $p_i$).Then the image of $D_\Gamma$ in $\overline{M}_{0,n}$ is a point, and $D_\Gamma$ isomorphic to a toric surface obtained by blowing up $\mathbb{P}^2$ at three points. In particular, $D_\Gamma$ has six torus-invariant $(-1)$-curves $C_1,\ldots,C_6$ (each of which is a boundary stratum corresponding to a degeneration of the bottom level of $\Gamma$), but its Picard rank is 4. So, there are two independent relations between $C_i$'s that cannot be written as a pullback of a WDVV relations in $\overline{M}_{0,7}$.
\end{example}
\section{Smooth coarse moduli spaces}\label{smooth_spaces}
In this section, we will analyze the cases where the coarse moduli space $\overline{B}$ of $\overline{\mathcal{B}}=\mathbb{P}\Xi\overline{M}_{0,n}(\mu)$ is smooth, where $\mu=(m_1,\ldots,m_n)$. Throughout this section, we will assume that $m_1\geq m_2\geq\ldots\geq m_n$. Recall from \cite[Sec. 2]{ccm22} that the isotropy group of a multiscale differential $(C,\omega)$ in $\overline{\mathcal{B}}$ is given by an extension of the group of curve automorphisms $\text{Aut}(C,\omega)$ by the group of ghost automorphisms $K_\Gamma=\text{Tw}_\Gamma/\text{Tw}^s_\Gamma$. Here, $\text{Tw}_\Gamma$ is the twist group, a discrete subgroup of $\mathbb{C}^L$ such that the quotient $\mathbb{C}^L/\text{Tw}_\Gamma$ is naturally identified with the level rotation torus. In particular, $\text{Tw}_\Gamma$ acts on the multiscale differentials compatible with $\Gamma$ by fixing the differentials on each level and bringing the prongs back to themselves, whereas the simple twist group $\text{Tw}^s_\Gamma$ is the subgroup of $\text{Tw}_\Gamma$ consisting of elements that can be written as a product of twists that act on one level passage only. Since $\text{Aut}(C,\omega)$ is trivial for a curve of genus 0, it follows that the singularities of $\overline{B}$, if any, are induced by the ghost automorphisms.

First, we introduce the following definition:
\begin{definition}\label{realizable_def}
  An enhanced level graph $\Gamma$ is said to be \textit{realizable (with in a given stratum of differentials)} if there is a multiscale differential $\omega$ compatible with $\Gamma$ (that is, $D_\Gamma\neq \emptyset$). When genus is 0, this entails with a collection of differentials $\{\omega_v\}_{v\in V(\Gamma)}$ (i.e. a multiscale differential) has prescribed orders of zeros and poles (including at the nodes), such that the two ends of a horizontal edge have simple poles with opposite residues and the level structure of $\Gamma$ is compatible -- for a vertical edge, the multiscale differential has non-negative order on the downward pointing half edge, and negative order on the upward pointing half edge, such that the sum of orders is $-2$.\\
  Otherwise, we call $\Gamma$ \textit{unrealizable}.
\end{definition}
The cherry graphs in \cref{cherry_figure} will play important role in our discussion below, so we introduce the following notations:

\textbf{Notation:} For brevity and convenience, we will denote the cherry graph \cref{regular_cherry} by the partition $\langle i_1,\ldots,i_k\,||\,i_{k+1},\ldots,i_\ell\,|\,i_{\ell+1},\ldots,i_n\rangle$ and the upside down cherry \cref{inverted_cherry} by the partition $\langle i_{k+1},\ldots,i_\ell\,|\,i_{\ell+1},\ldots,i_n\,||\,i_1,\ldots,i_k\rangle$ (double vertical lines $||$ are used to delineate the marked legs on the root vertex of the graph).
\begin{figure}[H]
   \centering
   \begin{subfigure}{.5\textwidth}
     \centering
  \begin{tikzpicture}[main/.style = {draw, circle, fill=black}] 
  \node[main, minimum size=2mm,pin=135:$i_1$,pin=45:$i_k$] (1) at (0,0) {};
  \node[main, minimum size=2mm,pin=250:$i_{\ell+1}$,pin=340:$i_n$] (2) at (1,-1) {};
  \node[main, minimum size=2mm,pin=200:$i_{k+1}$,pin=290:$i_\ell$] (3) at (-1,-1) {};
  \draw (1) -- node[right] {b} ++ (2);
  \draw (1) -- node[left] {a} ++ (3);
\end{tikzpicture}
\caption{A cherry}\label{regular_cherry}
\end{subfigure}%
\begin{subfigure}{.5\textwidth}
  \centering
  \begin{tikzpicture}[main/.style = {draw, circle, fill=black}]
  \node[main, minimum size=2mm,pin=225:$i_1$,pin=315:$i_k$] (1) at (0,0) {};
  \node[main, minimum size=2mm,pin=110:$i_{\ell+1}$,pin=20:$i_n$] (2) at (1,1) {};
  \node[main, minimum size=2mm,pin=160:$i_{k+1}$,pin=70:$i_\ell$] (3) at (-1,1) {};
  \draw (1) -- node[right] {b} ++ (2);
  \draw (1) -- node[left] {a} ++ (3);
\end{tikzpicture}
\caption{An upside down cherry}\label{inverted_cherry}
\end{subfigure}
\caption{Cherry graphs}\label{cherry_figure}
\end{figure}
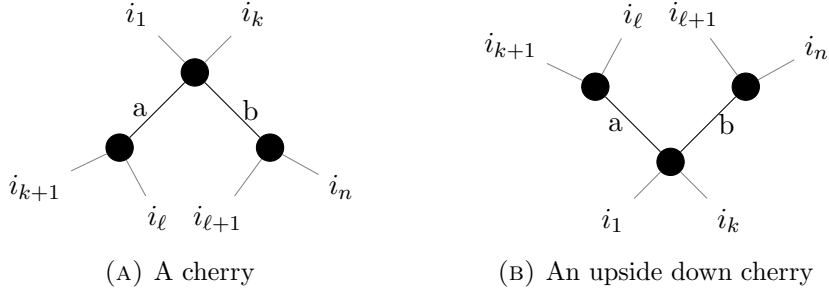

In either of the graphs in \cref{cherry_figure} above, we have the enhancements $a=|-1-m_{i_{k+1}}-\ldots-m_{i_\ell}|$ and $b=|-1-m_{i_{\ell+1}}-\ldots-m_{i_n}|$, so the enhancements are equal if and only if $m_{i_{k+1}}+\ldots+m_{i_\ell}=m_{i_{\ell+1}}+\ldots+m_{i_n}$. This leads to the following definition which will simplify our exposition:
\begin{definition}
  A cherry \cref{regular_cherry} or an upside down cherry \cref{inverted_cherry} will be called {\em balanced} if the enhancements on the two edges are equal. That is, for a cherry $\langle i_1,\ldots,i_k\,||\,i_{k+1},\ldots,i_\ell\,|\,i_{\ell+1},\ldots,i_n\rangle$ or an upside down cherry $\langle i_{k+1},\ldots,i_\ell\,|\,i_{\ell+1},\ldots,i_n\,||\,i_1,\ldots,i_k\rangle$, this means $m_{i_{k+1}}+\ldots+m_{i_\ell}=m_{i_{\ell+1}}+\ldots+m_{i_n}$.\\ Otherwise, we call the (possibly upside down) cherry {\em unbalanced}.
\end{definition}

The significance of the (un)balanced cherry will be elucidated by the computations in \cite[Example 5.7]{ccm22}, which we recall it here:
\begin{example}\cite[Example 5.7]{ccm22}\label{cherry_example}
  Suppose $\Gamma$ is a slanted cherry graph:
  \begin{figure}[H]
  \centering
  \begin{tikzpicture}[main/.style = {draw, circle, fill=black}] 
  \node[main] (1) at (0,0) {};
  \node[main] (2) at (-1,-1) {};
  \node[main] (3) at (1,-2) {};
  \draw (1) -- node[above left] {a} ++ (2);
  \draw (1) -- node[above right] {b} ++ (3);
\end{tikzpicture}.
\caption{A slanted cherry}\label{slanted_cherry}
\end{figure}
Assume that the enhancement on the shorter edge is $a$ and that on the longer edge is $b$. Then the index of $\text{Tw}_\Gamma^s$ in $\text{Tw}_\Gamma$ is given by $b/\text{gcd}(a,b)$ (which is thus also the order of the group of ghost automorphisms). In particular, if $a\neq b$ then the group of ghost automorphism is not trivial (if $b$ divides $a$, then we slant the cherry the other way to get an enhanced level graph with non-trivial group of ghost automorphisms). The situation for the upside down slanted cherry is analogous.
\end{example}

Thus by the preceding discussion and Example \ref{cherry_example}, we can conclude:
\begin{lemma}\label{equal_enhancements}
  If the set $\{1,\ldots,n\}$ admits a partition \[\langle i_1,\ldots,i_k\,|\,i_{k+1},\ldots,i_\ell\,|\,i_{\ell+1},\ldots,i_n\rangle\] such that the cherry $\langle i_1,\ldots,i_k\,||\,i_{k+1},\ldots,i_\ell\,|\,i_{\ell+1},\ldots,i_n\rangle$ or the upside down cherry $\langle i_{k+1},\ldots,i_\ell\,|\,i_{\ell+1},\ldots,i_n\,||\,i_1,\ldots,i_k\rangle$ is realizable and unbalanced (i.e. $m_{i_{k+1}}+\ldots+m_{i_\ell}\neq m_{i_{\ell+1}}+\ldots+m_{i_n}$), then the coarse moduli space $\overline{B}$ of $\mathbb{P}\Xi\overline{M}_{0,n}(\mu)$ is {\em not} a smooth variety.
\end{lemma}
Thus, our strategy is to find a {\em realizable} and {\em unbalanced} (possibly upside down) cherry to eliminate the cases where the coarse moduli space cannot be smooth, and then verify that the few cases that are left are actually smooth.

Before we discuss the smooth cases, let us note the following:
  \begin{lemma}\label{realizable}
\begin{enumerate}
\item The cherry graph $\langle i_1,\ldots,i_k\,||\,i_{k+1},\ldots,i_\ell\,|\,i_{\ell+1},\ldots,i_n\rangle$ is realizable if and only if $m_{i_1}+\ldots+m_{i_k}\leq -2$ and $m_{i_{k+1}}+\ldots+m_{i_\ell}\geq 0$ and $m_{i_{\ell+1}}+\ldots+m_{i_n}\geq 0$.
\item The upside down cherry graph $\langle i_{k+1},\ldots,i_\ell\,|\,i_{\ell+1},\ldots,i_n\,||\,i_1,\ldots,i_k\rangle$ is realizable if and only if $m_{i_1}+\ldots+m_{i_k}\geq 2$ and $m_{i_{k+1}}+\ldots+m_{i_\ell}\leq -2$ and $m_{i_{\ell+1}}+\ldots+m_{i_n}\leq -2$.
\end{enumerate}
\end{lemma}
\begin{proof}
  We will outline the proof for the cherry graph \cref{regular_cherry} (i.e. part (1)); the other case is analogous. For each edge, a multiscale differential has non-negative order on the downward pointing half edge (attached to the top vertex). Since the sum of orders of zeros and poles should sum to $-2$, this implies the sum of legs on the top vertex can be at most $-2$, i.e. $m_{i_1}+\ldots+m_{i_k}\leq -2$. On the other hand, for each edge (say, the left one), a multiscale differential has order at most $-2$ on the the upward pointing half edge (attached to the bottom left vertex). Since the sum of orders of zeros and poles should again be $-2$ for the bottom left vertex, we conclude $m_{i_{k+1}}+\ldots+m_{i_\ell}\geq 0$. Analogously, for the bottom right vertex, we obtain $m_{i_{\ell+1}}+\ldots+m_{i_n}\geq 0$.
\end{proof}
This lemma will be used repeatedly in our discussion throughout the rest of this section.

Now, we move onto determining all the smooth cases.
\subsection{Smoothness when $n\geq 7$}
Our objective is to prove that when $n\geq 7$, the only cases when the coarse moduli spaces are smooth are $\mu=(0^n,-2)$ and $(0^n,-1^2)$. First we prove that these are the only cases that could potentially be smooth:
\begin{enumerate}
\item \textbf{At least four of $\bm{m_i}$'s are non-negative.} We suppose $m_1,\ldots,m_4$ are non-negative. Consider a cherry graph $\Gamma_1 = \langle 5,6,\ldots,n\,||\,1,2\,|\,3,4\rangle$, which is realizable by \cref{realizable} since $m_5+\ldots+m_n=-2-m_1-\ldots-m_4\leq -2$. Since $m_1\geq m_2\geq m_3\geq m_4$, any such graph is balanced only if $m_1=m_2=m_3=m_4$. In fact, swapping one of these marked points with any other with non-negative order, we conclude that all the non-negative $m_i$'s are equal, that is, $m_1=m_2=\ldots=m_k\eqqcolon m$, where $m_k\geq 0>m_{k+1}$. If $k\geq 5$ then we consider a new cherry $\Gamma_1'=\langle 6,\ldots,n\,||\,1,2,5\,|\,3,4\rangle$; if it is still balanced then we must have $m_5=m=0$, that is, all non-negative entries are 0. Thus, for degree reason, we must have $\mu=(0^{n-1},-2)$ or $(0^{n-2},-1^2)$.

  Now, we assume $k=4$ instead, that is, $m=m_4\geq 0>m_5$. We still consider the graph $\Gamma_1'$ which is balanced only if $2m=2m+m_5$, which cannot happen (since $m_5<0$). So, by \cref{equal_enhancements}, $\Gamma_1'$ is unrealizable, which can only happen if $2m+m_5$ is negative. In this case, we consider an upside down cherry $\Gamma_2=\langle 5,6\,|\,4,7,8,\ldots,n\,||\,1,2,3\rangle$. So either \textit{(i)} the upside down cherry unrealizable, or \textit{(ii)} it is balanced. Since $m+m_7+\ldots+m_n\leq 2m+m_5\leq -1$, $\Gamma_2$ is unrealizable only when $n=7,\ m_5=m_7$ and $m=0$. But since $m_5,m_6,m_7$ are all negative, this cannot happen for degree reason (the sum of $m_i$'s would be less than $-2$). If $\Gamma_2$ is balanced then $m_5+m_6=m+m_7+\ldots+m_n$. In this case, consider a new upside down cherry $\Gamma_2'=\langle 5,6\,|\,3,4,7,8,\ldots,n\,||\,1,2\rangle$. If it is balanced (which happens only when $2m+m_7+\ldots+m_n=m_5+m_6$), then $m=0$ and $m_5$ through $m_n$ are negative, which cannot happen for degree reasons. So, $\Gamma_2'$ must be unrealizable, which, by lemma \cref{realizable}, can occur only when $2m+m_7+\ldots+m_n\geq -1$. But $2m+m_7\leq 2m+m_5\leq -1$, so this can happen only when $n=7$ and $2m+m_7=2m+m_5=-1$, that is, $m_5=m_6=m_7\eqqcolon\alpha$. But then $m_1+\ldots+m_7=4m+3\alpha =-2+\alpha<-2$, so cannot occur for degree reason.
  
Thus, when at least four $m_i$'s are non-negative, $\overline{B}$ could be smooth only in the cases $\mu=(0^{n-1},-2)$ and $(0^{n-2},-1^2)$.
\item \textbf{At most three $m_i$'s are non-negative.} That is, we assume $m_4<0$. We consider an upside down cherry graph $\Lambda_1=\langle n-3,n-2\,|\,n-1,n\,||\,1,2,\ldots,n-4\rangle$. Then the enhancements of the two edges are equal only when $m_n+m_{n-1}=m_{n-2}+m_{n-3}$, which can happen only when $m_n=m_{n-1}=m_{n-2}=m_{n-3}\eqcolon m$. In fact, replaceing one of the four marked points at the top level with any other with negative order $m_i$, we can conclude that all the negative $m_i$'s are equal. Furthermore, if $n>7$, we can consider the upside down chery given by $\langle 4,n-3,n-2\,|\,n-1,n\,||\,1,2,3,5,\ldots,n-4\rangle$, which is realizable and unbalanced, so we necessarily have $n=7$. So, we assume $n=7$. Then we instead look at $\Lambda_2=\langle 3,n-3,n-2\,|\,n-1,n\,||\,1,2,4,\ldots,n-4\rangle$. If $\Lambda_2$ is a realizable upside down cherry, it must be balanced, so $m_3=0$ (since $m_n+m_{n-1}=m_{n-2}+m_{n-3}=2m$). Replacing marked point 3 by 1 or 2, we necessarily have $m_1=m_2=m_3=0$, so the sum of $m_i$'s is less than -2, which cannot happen. So, $\Lambda_2$ is unrealizable, which implies $2m+m_3\geq -1$. In this case, $-2=4m+m_1+m_2+m_3\geq 4m+2m_3+m_1\geq -2+m_1\geq -2$, so each inequality is an equality, which implies $m_1=m_2=m_3=0$, which is again not possible for degree reason.

  Thus, if at most three $m_i$'s are non-negative and $n\geq 7$ then $\overline{B}$ cannot be smooth.
\end{enumerate}
From this discussion, we can see that the only cases for which $\overline{B}$ might be smooth are $\mu=(0^{n-1},-2)$ and $(0^{n-2},-1^2)$. We will prove that these two cases are indeed smooth:
\begin{proposition}\label{smooth_geq7}
  If $n\geq 7$ then the coarse moduli space $\overline{B}$ of $\mathbb{P}\Xi\overline{M}_{0,n}(\mu)$ is smooth if and only if $\mu=(0^{n-1},-2)$ or $(0^{n-2},-1^2)$.
\end{proposition}
\begin{proof}
  The ``only if'' part has already been proven. So, we will next prove that in these two cases, the coarse moduli space $\overline{B}$ is smooth. Note that in the case of $\mu=(0^{n-1},-2)$, all level graphs that can appear are rooted level trees, with the marked point of order $-2$ necessarily present in the root vertex (that is, the unique top level vertex). All other half edges appearing at the root vertex will have order 0; in particular, all edges going down from the root vertex will have enhancement equal to 1. Each such edge will lead to a branch of the tree, which will again be a rooted level tree with unique half-edge of order $-2$, and all other half edges of order 0. So, we can now proceed recursively to the lower levels to conclude that the enhancements on all the edges of the level graph have to be equal to 1. Consequently, the twist group and the simple twist group are trivially equal and thus the group of ghost automorphisms is trivial for every level graph.
  
  The case for $\mu=(0^{n-2},-1^2)$ is analogous -- the two marked points of negative order should necessarily be in the top level, but now the top level might have up to two vertices, connected by a horizontal edge. However, all the branches of the tree going down from the top level will still necessarily be rooted level trees with unique half-edge of order $-2$, so the same argument as above can be applied to conclude that the group of ghost automorphisms is trivial.
\end{proof}

\subsection{Smoothness when $n=6$}
When $n=5,6$, we lose the flexibility in the choice of the cherry graphs that we had before. Because of this, there are more cases where the coarse moduli space is smooth. In this subsection, we determine all the cases where the coarse moduli space $\overline{B}$ of $\mathbb{P}\Xi\overline{M}_{0,6}(\mu)$ is smooth.
\begin{enumerate}
\item \textbf{Exactly one $m_i$ is non-negative.} That is, $m_1\geq 0>m_2$. Then the upside down cherry graphs $\langle 3,4\,|\,5,6\,||\,1,2\rangle$ and $\langle 2,3,4\,|\,5,6\,||\,1\rangle$ are both realizable and at least one is unbalanced. (They are realizable by \cref{realizable} because $m_1\geq m_1+m_2=-2-m_3-\ldots-m_6\geq 2$.)
\item \textbf{Exactly one $m_i$ is negative.} That is, $m_5\geq 0> m_6$. Then the cherry $\langle 6\,||\,1,2,3\,|\,4,5\rangle$ is realizable ($m_6\leq -2$ since $\sum_im_i=-2$), so it must be balanced: $m_4+m_5=m_1+m_2+m_3\geq m_1+m_4+m_5$, so $m_1=0$. Since $m_1\geq\ldots\geq m_5\geq 0$, we can conclude that $m_1=\ldots=m_5=0$, so $\mu=(0^5,-2)$.
\item \textbf{Exactly two $m_i$'s are non-negative.} That is, $m_2\geq 0> m_3$. Then the upside down cherry graph $\Gamma_1=\langle 3,4\,|\,5,6\,||\,1,2\rangle$ is obviously realizable, so it must be balanced: $m_3+m_4=m_5+m_6$, which implies $m_3=\ldots=m_6\eqcolon m$. Next, consider a new upside down cherry $\Gamma_2=\langle 2,3,4\,|\,5,6\,||\,1\rangle$. Then either \textit{(i)} $\Gamma_2$ is unrealizable, or \textit{(ii)} it is balanced. The case \textit{(i)} can occur only when $m_1=1$ or $2m+m_2\geq -1$. In case of the former, $m_1=m_2=1$ and $m_3=\ldots=m_6=-1$ for degree reasons, so $\mu=(1^2,-1^4)$. In case of the latter, note that $2m+m_1\geq 2m+m_2\geq -1$, so $4m+m_1+m_2=-2$ forces $2m+m_1=2m+m_2=-1$ (in particular, $m_1=m_2$), which implies $\mu=(a^2,-b^4)$ for $a-2b=-1$ and $b\geq 1$.

  On the other hand, if the case \textit{(ii)} occurs then $m_2=0$. Then we look at a level graph $\Gamma_3=\langle 3,4,5\,|\,2,6\,||\,1\rangle$. So, the enhancements of the two edges are $m_6+m_2=m_6=m$ and $3m$, which are different. So, for smoothness of $\overline{B}$, the resulting upside down cherry cannot be realizable, so we must have $m_6+m_2=m=-1$, and thus $m_1=2$, that is, $\mu=(2,0,-1^4)$.
\item \textbf{Exactly two $m_i$'s are negative.} The argument is similar to part (3), but we swap the role of cherry graph and upside down cherry graph. We will obtain that either $\mu=(a^4,-b^2)$ for $2a-b=-1$ and $b\geq 1$ or $\mu=(0^4,-1^2)$.
\item \textbf{Exactly three $m_i$'s are non-negatve.} That is, $m_3\geq 0> m_4$. First we consider the upside down cherry graph $\Lambda_1=\langle 3,6\,|\,4,5\,||\,1,2\rangle$. Then for smoothness of $\overline{B}$, either \textit{(a)} $\Lambda_1$ is unrealizable, or \textit{(b)} $\Lambda_1$ is balanced. $\Lambda_1$ is unrealizable occurs when either $m_1+m_2\leq 1$ or $m_6+m_3\geq -1$ (see \cref{realizable}). In case of the former, we must have $m_1+m_2=1$ (for degree reason), so $m_1=1$ and $m_2=m_3=0$, which implies $\mu=(1,0^2,-1^3)$. So, let's assume $m_6+m_3\geq -1$. In fact, we can assume that $m_6+m_3=-1$ (otherwise we would have $-2=m_1+\ldots+m_6\geq 3(m_3+m_6)\geq 0$), and also that $m_4+m_1\geq 0$ (otherwise $\sum_im_i\leq 3(m_1+m_4)\leq -3$). Now, we consider the cherry graph $\Lambda_2=\langle 5,6\,||\,1,4\,|\,2,3\rangle$. Then the cherry is realizable (see \cref{realizable}), so it must be balanced, which implies $m_1+m_4=m_2+m_3$. If $m_1+m_4=0$ then we must have $m_2=m_3=0$, and thus $m_6=-1-m_3=-1=m_5=m_4$, and consequently, $\mu=(1,0^2,-1^3)$. On the other hand, if $m_1+m_4=m_2+m_3\geq 1$ then $m_2\geq 1$, so that $m_2+m_3+m_5\geq m_2+m_3+m_6\geq m_2-1\geq 0$. Thus, the new cherry level graph $\Lambda_2'=\langle 6\,||\,1,4\,|\,2,3,5\rangle$ is realizable and unbalanced (so the coarse space cannot be smooth).

  Next, we look at the case \textit{(b)}, that is, $\Lambda_1$ is balanced (i.e., $m_6+m_3=m_5+m_4\leq -2$). Then we consider the upside down cherry graph $\Lambda_1'=\langle 2,3,6\,|\,4,5\,||\,1\rangle$. Then either \textit{(i)} $\Lambda_1'$ is unrealizable, or \textit{(ii)} $\Lambda_1'$ is balanced. By \cref{realizable}, $\Lambda_1'$ is unrealizables when $m_1\leq 1$ (which implies $m_1=1$ for degree reason) or $m_2+m_3+m_6\geq -1$. In either case, we consider the cherry graph $\Lambda_2=\langle 5,6\,||\,1,4\,|\,2,3\rangle$. If $m_1=1$ and $m_2=0$ then $\mu=(1,0^2,-1^3)$. Similarly, if $m_1=m_2=1$ then $m_4+m_5+m_6=-2-m_1-m_2-m_3\geq -5$, so $m_4=-1$ (since $m_4\leq m_5\leq m_6$), so the two edges of $\Lambda_2$ will have different enhancements ($m_1+m_4=0\neq m_2+m_3$).\\
  So, now we consider $m_2+m_3+m_6\geq -1$. Then for $\Lambda_2$ to not induce singularity on $\overline{B}$, either $m_1+m_4\leq -1$ (unlike in case \textit{(a)}, $\Lambda_2$ can be unrealizable) or it is balanced (so $m_2+m_3=m_1+m_4$). In case of the former, $m_3+m_6\leq m_2+m_5\leq m_1+m_4\leq -1$, so $m_1+\ldots+m_6\leq -3$, which is impossible. So, we assume $m_2+m_3=m_1+m_4$. We have equalities: $m_6+m_3=m_5+m_4$ and $m_2+m_3=m_1+m_4$, as well as an inequaliy $m_2+m_3+m_6\geq -1$ (which implies $m_1+m_4+m_5=-2-(m_2+m_3+m_6)\leq -1$). Adding the two equalities, we obtain $m_3-1\leq m_6+m_2+2m_3=m_1+m_5+2m_4\leq m_4-1$, so that $m_3-m_4\leq 0$, which is a contradiction.

  Thus, if exactly three $m_i$'s are non-negative and $\overline{B}$ is smooth then $\mu=(1,0^2,-1^3)$.
\end{enumerate}
From this discussion, we see that the only cases for which $\overline{B}$ might be smooth are $\mu=(0^5,-2)$, $(2,0,-1^4)$, $(1,0^2,-1^3)$ and $(a^4,b^2)$ for $2a+b=-1$ (where $a,b$ are integers). In fact, as in \cref{smooth_geq7}, we can prove that each of these cases are smooth. For this, we will need to prove that $\overline{B}$ is not singular along other boundary strata of codimension $\geq 2$. To this end, we recall the description of $\text{Tw}_\Gamma$ and $\text{Tw}_\Gamma^s$ as given in \cite[Section 5]{ccm22} (especially, see the discussion right before example 5.7 there). Given an enhanced level graph $\Gamma$ with $L$ levels, define a lattice $M'\subset \mathbb{R}^L$ as:
\[M'=\left\langle\frac{1}{\ell_i}w_i:\,i=1,\ldots,L\right\rangle_{\mathbb{Z}}\subset\mathbb{R}^L,\]where $w_i$ is the $i^{\text{th}}$ unit vector in $\mathbb{R}^L$ and $\ell_i$ is the lcm of all enhancements of the edges crossing the $i^{\text{th}}$ level passage. Then we can realize the simple twist group as the dual of this lattice: $\text{Tw}_\Gamma^s=(M')^\vee$. On the other hand, using the notation $w^i_j=\sum_{k=i+1}^jw_k\in\mathbb{R}^L$, we can realize the twist group as $\text{Tw}_\Gamma=M^\vee$, such that
\[M=\left\langle\frac{1}{\kappa_e}w^{e^+}_{e^-},\ e\in E(\Gamma)\right\rangle_{\mathbb{Z}},\]
where $e^\pm$ are the upper and lower ends of the edge $e$ and $\kappa_e$ is the enhancement on $e$. Thus, to prove that $K_\Gamma$ is trivial, we need to prove that $M=M'$ for each enhanced level graph of codimension $\geq 2$. We will now prove
\begin{proposition}\label{smooth_eq6}
  The coarse moduli space $\overline{B}$ of $\mathbb{P}\Xi\overline{M}_{0,6}(\mu)$ is smooth exactly when $\mu=(0^5,-2)$, $(2,0,-1^4)$, $(1,0^2,-1^3)$ or $(a^4,b^2)$ for integers $a,b$ satisfying $2a+b=-1$.
\end{proposition}
\begin{proof}
  We have already checked that the group $K_\Gamma$ is trivial for all slanted cherries and slanted upside down cherries appearing in the each of the four cases. Note that in each of the four cases, the exceptional divisors are given by $D_\Gamma$ where $\Gamma$ is either a cherry or an upside down cherry. Also, observe that if the $i^{\text{th}}$ level passage of a level graph $\Gamma$ has exactly one edge $e$ crossing it such that $e$ connects two vertices in successive levels then $w_{e^-}^{e^+}/\kappa_e=w_i/\ell_i$ and for any other edge $f$, the $i^{\text{th}}$ component of the vector $w_{f^-}^{f^+}$ is 0, so $M$ and $M'$ are equal for $\Gamma$ if they are equal for its undegeneration collapsing the $i^{\text{th}}$ level passage. Thus we can always assume that if the level graph $\Gamma$ has a level passage with exactly one edge then its ends are not in the successive levels, so we only need to check triviality of $K_\Gamma$ where $\Gamma$ is one of the following (or upside down versions of the following):
  
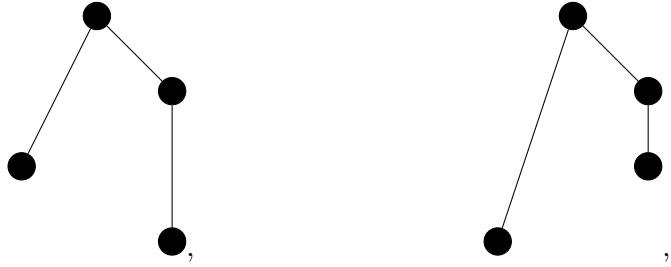
\begin{figure}[H]
    \centering
    \begin{subfigure}{.5\textwidth}
     \centering
\begin{tikzpicture}[main/.style = {draw, circle, fill=black}] 
  \node[main] (1) at (0,0) {};
  \node[main] (2) at (1,-1) {};
  \node[main] (3) at (-1,-2) {};
  \node[main] (4) at (1,-3) {};
  \draw (1) to (2);
  \draw (1) to (3);
  \draw (2) to (4);
\end{tikzpicture},
\end{subfigure}%
\begin{subfigure}{.5\textwidth}
     \centering
\begin{tikzpicture}[main/.style = {draw, circle, fill=black}] 
  \node[main] (1) at (0,0) {};
  \node[main] (2) at (1,-1) {};
  \node[main] (3) at (-1,-3) {};
  \node[main] (4) at (1,-2) {};
  \draw (1) to (2);
  \draw (1) to (3);
  \draw (2) to (4);
\end{tikzpicture},
\end{subfigure}
\end{figure}
such that enhancements of the edges are equal (otherwise there will be an undegeneration yielding a cherry with different enhancements). We have suppressed the marked points for simplicity. Then we can check \'a la \cite[Example 5.7]{ccm22} that for each of these graphs, $M$ and $M'$ are equal.
\end{proof}
\subsection{Smoothness when $n=5$}
In order to state our proposition more concisely, we will allow $m_i$'s to not necessarily be in descending order. However, in our proof, we will return back to assuming $m_1\geq\ldots\geq m_5$.
\begin{proposition}\label{smooth_nequals5}
  The following are all (up to permutation of $m_i$) the cases for which the coarse moduli space $\overline{B}$ of $\mathbb{P}\Xi\overline{M}_{0,5}(\mu)$ is smooth ($a,b\in\mathbb{Z}$):
  \begin{enumerate}
  \item $\mu=(2a-1,a-1,-a^3)$ or $(a^2,0,-a-1^2)$. In these cases, $\overline{B}$ is isomorphic to $\overline{M}_{0,5}$.
\item $\mu=(4a-2,a-1,-a^2,-3a+1)$. These are all isomorphic to a blowup of $\overline{M}_{0,5}$ at one point.
\item $\mu=(3a-1,2a-1,-a,-2a^2)$ and $(2a-1^2,-a^2,-2a)$. Except when $a=0$ (for which  $\overline{B}\cong \overline{M}_{0,5}$), these are isomorphic to a blowup of $\overline{M}_{0,5}$ at two points.
\item $\mu=(4a-2,-a^4)$. These are isomorphic to a blowup of $\overline{M}_{0,5}$ at three points.
\item $\mu=(4a-2^2,-a^2,-6a+2)$. These are isomorphic to a blowup of $\overline{M}_{0,5}$ at four points.
\item $\mu=(a^2,b^3)$ with $2a+3b=-2$. These are isomorphic to a blowup of $\overline{M}_{0,5}$ at six points.
  \end{enumerate}
\end{proposition}
Since the dual graphs on $\overline{M}_{0,5}$ can have at most three vertices, the exceptional divisors of $\overline{B}$ over $\overline{M}_{0,5}$ are all given by enhanced level graphs with two levels such that at least one of the levels must have two disjoint vertices (so they should be either a cherry or an upside down cherry). So to prove each of these cases is smooth, we simply write down the possible cherries/upside down cherries for all these cases and check that they are balanced.

The following will be convenient in the proof of \cref{smooth_nequals5}:
  \begin{claim}\label{all_balanced_cherries}
    If $m_1\geq\ldots\geq m_5$ with $m_2\geq 0$ and $m_4<0$ then:
    \begin{enumerate}[(a)]
    \item all possibilities for realizable and  balanced cherries are $\langle 5\,||\,1,4\,|\,2,3\rangle$, $\langle 5\,||\,1,3\,|\,2,4\rangle$, $\langle 4\,||\,1,5\,|\,2,3\rangle$, $\langle 4\,||\,1,3\,|\,2,5\rangle$, $\langle 3\,||\,1,5\,|\,2,4\rangle$ and $\langle 3\,||\,1,4\,|\,2,5\rangle$.
    \item all possibilities for realizable and  balanced upside down cherries are $\langle 2,5\,|\,3,4\,||\,1\rangle$, $\langle 2,4\,|\,3,5\,||\,1\rangle$, $\langle 1,5\,|\,3,4\,||\,2\rangle$, $\langle 1,4\,|\,3,5\,||\,2\rangle$, $\langle 1,5\,|\,2,4\,||\,3\rangle$ and $\langle 1,4\,|\,2,5\,||\,3\rangle$.
    \end{enumerate}
  \end{claim}
  \begin{proof}
    Any other possibility can be checked easily to be either unrealizable or unbalanced. For instance, the cherry $\langle 5\,||\,1,2\,|\,4,3\rangle$ is unbalanced because $m_1+m_2\geq m_3>m_3+m_4$. Other possibilities can be excluded analogously.
  \end{proof}
  \begin{remark}\label{realizable_implication}
    Let us note that in part (a) (resp. part (b)) of \cref{all_balanced_cherries}, if any of the last three cherries (resp. upside down cherries) is realizable then the first three are all realzable. This can be checked easily as a consequence of \cref{realizable}.
    \end{remark}
\begin{proof}[Proof of \cref{smooth_nequals5}]
  Note that since enhancements on a dual graph is automatically determined (\cref{enhancements_on_tree}), there cannot be more than one enhanced level graph of codimension 1 over a dual graph with two edges. So, $\overline{B}$ is a blowup of $\overline{M}_{0,5}$ at finitely many (reduced) points. Our strategy will then be to analyze all the cases for which $\overline{B}$ is a blowup of $i$ points in $\overline{M}_{0,5}$.

  Let us observe that if $\mu=(m_1,\ldots,m_5)$ with exactly four entries non-negative (resp. negative) then we always obtain a realizable cherry graph (resp. upside down cherry graph) where the marked legs of non-negative (resp. negative) orders are distributed on the two leaves. Imposing that they are balanced and permuting the marked legs on the leaves, we see that the non-negative (resp. negative) $m_i$'s have to be equal, so $\mu$ has to be $(a^4,-4a-2)$ (resp. $(4a-2, -a^4)$). So, for the rest of the argument, we will assume that at least two entries of $\mu$ are non-negative and at least two are negative. Additionally, as mentioned in the beginning of the section, we will assume $m_1\geq m_2\geq\ldots\geq m_5$ throughout the proof. So, by inspection of the inequalities in \cref{realizable}, if a cherry graph is realizable on $\overline{B}$ then one of realizable cherries must be $\Gamma_1$ (\cref{gamma1}); similarly, if an upside down cherry graph is realizable on $\overline{B}$ then one of the realizable upside down cherries must be $\Gamma_2$ (\cref{gamma2}).
  \begin{enumerate}
  \item \textbf{$\overline{B}$ is isomorphic to $\overline{M}_{0,5}$.} This means none of the stable cherry or upside down cherry graphs will be realizable. For instance, if we take a cherry graph $\Gamma_1=\langle 5\,||\,1,4\,|\,2,3\rangle$ (\cref{gamma1}) then (by \cref{realizable}) either $m_5=-1$ or $m_1+m_4\leq -1$ or $m_2+m_3\leq -1$. Similarly, for an upside down cherry $\Gamma_2=\langle 2,5\,|\,3,4\,||\,1\rangle$ (\cref{gamma2}), either $m_1\in \{0,1\}$ or $m_2+m_5\geq -1$ or $m_3+m_4\geq -1$. Of course $m_i$'s are additionally subject to the degree constraint: $\sum_im_i=-2$.
\begin{figure}[H]
   \centering
   \begin{subfigure}{.5\textwidth}
     \centering
  \begin{tikzpicture}[main/.style = {draw, circle, fill=black}] 
  \node[main, minimum size=2mm,pin=90:$5$] (1) at (0,0) {};
  \node[main, minimum size=2mm,pin=225:$1$,pin=315:$4$] (2) at (-1,-1) {};
  \node[main, minimum size=2mm,pin=225:$2$,pin=315:$3$] (3) at (1,-1) {};
  \draw (1) to (2);
  \draw (1) to (3);
\end{tikzpicture}
\caption{$\Gamma_1$}\label{gamma1}
\end{subfigure}%
\begin{subfigure}{.5\textwidth}
  \centering
  \begin{tikzpicture}[main/.style = {draw, circle, fill=black}]
  \node[main, minimum size=2mm,pin=270:$1$] (1) at (0,0) {};
  \node[main, minimum size=2mm,pin=135:$3$,pin=45:$4$] (2) at (1,1) {};
  \node[main, minimum size=2mm,pin=135:$2$,pin=45:$5$] (3) at (-1,1) {};
  \draw (1) to (2);
  \draw (1) to (3);
\end{tikzpicture}
\caption{$\Gamma_2$}\label{gamma2}
\end{subfigure}
\caption{Possible exceptional divisors in $\overline{B}$}
\label{exceptionals}
\end{figure}
The cases $m_5=-1$ and $m_1\in \{0,1\}$ are easily dealt with by enumerating all possible 5-tuples $\mu$ (noting that $m_1\geq\ldots\geq m_5$), and checking which one of them give smooth coarse moduli space $\overline{B}$ -- the smooth ones are $(1^3, -2, -3),\ (1^2,0,-2^2),\ (1,0,-1^3)$, and $(0^3,-1^2)$, which are all special cases of tuples listed in part (1) of the proposition.

So, we are left with four cases from above: $m_1+m_4\leq -1$ and $m_2+m_5\geq -1$, or $m_1+m_4\leq -1$ and $m_3+m_4\geq -1$, or $m_2+m_3\leq -1$ and $m_2+m_5\geq -1$, or $m_2+m_3\leq -1$ and $m_3+m_4\geq -1$. First, assume $m_1+m_4\leq -1$ and $m_2+m_5\geq -1$. Since $-1\geq m_1+m_4\geq m_1+m_5\geq m_2+m_5\geq -1$, we see that $m_1+m_4=m_1+m_5=m_2+m_5=-1$ (and thus $m_1=m_2$ and $m_4=m_5$). Additionally, $\sum_im_i=-2$ implies $m_3=0$, so $\mu=(a^2,0,-a-1^2)$.

The other cases can also be dealt with analogously -- for example, $m_1+m_4\leq -1$ and $m_3+m_4\geq -1$ implies $\mu=(a^3,-a-1,-2a-1)$ (which, after permutation and change of variable $a\rightarrow -a$, is the first tuple in part (1) of the proposition); $m_2+m_3\leq -1$ and $m_2+m_5\geq -1$ implies $\mu=(2a-1,a-1,a^3)$; and $m_2+m_3\leq -1$ and $m_3+m_4\geq -1$ implies $\mu=(a^2,0,-a-1^2)$.
\item \textbf{$\overline{B}$ is isomorphic to a blowup of $\overline{M}_{0,5}$ at a point.} The unique exceptional divisor is either represented by a balanced cherry or by a balanced upside down cherry. Suppose the upside down cherry $\Gamma_2=\langle 2,5\,|\,3,4\,||\,1\rangle$ (\cref{gamma2} from before) represents the exceptional divisor. Since it is balanced, we get $m_2+m_5=m_3+m_4\leq -2$ and $m_1\geq 2$. On the other hand, all other cherry and upside down cherry graphs must then be unrealizable. In particular, unrealizability of the cherry graph $\Gamma_1=\langle 5\,||\,1,4\,|\,2,3\rangle$ (\cref{gamma1}) implies either $m_1+m_4\leq -1$ or $m_2+m_3\leq -1$ (because $m_2+m_5\leq -2$, we have $m_5\leq -2$). On the other hand, $\Gamma_2'=\langle 2,4\,|\,3,5\,||\,1\rangle$ analogously yields another unrealizable upside down cherry graph, so we get $m_2+m_4\geq -1$.

  If $m_1+m_4\leq -1$ and $m_2+m_4\geq -1$ then $-1\leq m_2+m_4\leq m_1+m_4\leq -1$ implies $m_1=m_2$, which is not possible (otherwise swapping marked points 1 and 2 in $\Gamma_2$ would give another exceptional divisor on $\overline{B}$). So, we assume $m_2+m_3\leq -1$ and $m_2+m_4\geq -1$. Then $-1\leq m_2+m_4\leq m_2+m_3\leq -1$ implies $m_3=m_4$ (say equal to $-a$) and $m_2+m_3=-1$. So $m_2=a-1$ and $m_2+m_5=m_3+m_4$ implies $m_5=-3a+1$. And $\sum_im_i=-2$ implies $m_1=4a-2$, so that $\mu=(4a-2,a-1,-a^2,-3a+1)$.

  The argument for the case where exceptional divisor is given by regular cherry $\Gamma_1$ (\cref{gamma1}) is analogous, and implies $\mu=(3a+1,a^2,-a-1,-4a-2)$ (which, after a permutation and change of variable $a\rightarrow -a$, is the tuple in the part (2) of the proposition).
\item \textbf{$\overline{B}$ is isomorphic to a blowup of $\overline{M}_{0,5}$ at two points.} First we note that the two exceptional divisors should either both be given by regular cherry graphs or both by upside down cherry graphs. Indeed, if both $\Gamma_1=\langle 5\,||\,1,4\,|\,2,3\rangle$ and $\Gamma_2=\langle 2,5\,|\,3,4\,||\,1\rangle$ are realizable then $m_1+m_4=m_2+m_3$ an $m_2+m_5=m_3+m_4$. Then $\sum_im_i=-2$ implies $3m_3+m_2+m_4=-2$. On the other hand, swapping marked points 3 and 4 in $\Gamma_1$ (resp. swapping marked points 4 and 5 in $\Gamma_2$) must then yield an unrealizable level graph, so $m_2+m_4\leq -1$ (resp. $m_2+m_4\geq -1$). That is, $m_2+m_4=-1$, so putting this in $3m_3+m_2+m_4=-2$ implies $m_3=-1/3$, which is impossible.

  So, let us assume both exceptional divisors are represented by two cherry graphs $\Gamma_1$ and $\Gamma_1'$, where $\Gamma_1'$ is either $\langle 4\,||\,1,5\,|\,2,3\rangle$ or $\langle 5\,||\,2,4\,|\,1,3\rangle$. In case of the former, the latter is unrealizable, so that $m_2+m_4\leq -1$. And since both $\Gamma_1$ and $\Gamma_1'$ are balanced, we get $m_1+m_4=m_2+m_3\geq 0$ and $m_1+m_5=m_2+m_3$ (so that $m_4=m_5$). Additionally, since the graph $\Gamma_2$ is also unrealizable, either $m_2+m_5\geq -1$ or $m_3+m_4\geq -1$ or $m_1\leq 1$. If $m_2+m_5\geq -1$ then $-1\leq m_2+m_5=m_2+m_4\leq -1$ implies $m_2+m_4=-1$. So, putting $m_4=m_5=-m_2-1$ and $m_1=m_2+m_3-m_4$ in the equation $\sum_im_i=-2$ gives $m_2=-2m_3-1$. So, setting $m_3=-a$ then gives $\mu = (3a-1,2a-1,-a,-2a^2)$.

  The argument for the case where $\Gamma_1'=\langle 5\,||\,2,4\,|\,1,3\rangle$ is similar and yields $\mu=(2a-1^2,-a^2,-2a)$.

  Similarly, the argument for the case where both exceptional divisors are given by upside down cherry graphs is analogous, and yields $\mu=(2a^2,a,-2a-1,-3a-1)$ and $(2a,a^2,-2a-1^2)$, which differ from the tuples in part (3) of the proposition by a permutation and change of variable $a\rightarrow -a$.
\item \textbf{$\overline{B}$ is isomorphic to a blowup of $\overline{M}_{0,5}$ at three points.} We will show that when at least two $m_i$'s are negative and two are non-negative then we cannot get smooth $\overline{B}$ with exactly three exceptional divisors. Suppose the contrary, and first assume two of the exceptional divisors are given by cherry graphs $\Gamma_1=\langle 5\,||\,1,4\,|\,2,3\rangle$ and $\Gamma_1'$ and the third is given by upside down cherry $\Gamma_2=\langle 2,5\,|\,3,4\,||\,1\rangle$; the case where two are given by upside down cherry will be analogous. As in part (3) of the proof, the $\Gamma_1'$ can be either $\langle 4\,||\,1,5\,|\,2,3\rangle$ or $\langle 5\,||\,2,4\,|\,1,3\rangle$. Consequently, the equality of enhancements of the two edges for both $\Gamma_1$ and $\Gamma_1'$ implies $m_1+m_4=m_2+m_3$ and $m_4=m_5$ or $m_1=m_2$. If $m_4=m_5$, swapping 4 and 5 in $\Gamma_2$ gives fourth exceptional divisor, whereas if $m_1=m_2$ then swapping 1 and 2 in $\Gamma_2$ will give the fourth exceptional divisor.

  On the other hand, if all three exceptional divisors are given by cherry graphs, then by \cref{realizable_implication} they must be $\Gamma_1=\langle 5\,||\,1,4\,|\,2,3\rangle$, $\Gamma_1'=\langle 4\,||\,1,5\,|\,2,3\rangle$ and $\Gamma_1''=\langle 5\,||\,2,4\,|\,1,3\rangle$, so as before, we obtain $m_1+m_4=m_2+m_3$, $m_4=m_5$ and $m_1=m_2$ (so $m_3=m_5$ as well). That is, $\mu=(a^2,-b^3)$ for $2a-3b=-2$; but this case will have six exceptional divisors.

  If all three exceptional divisors are given by upside down cherries, the analogous argument implies that there are six exceptional divisors.
\item \textbf{$\overline{B}$ is isomorphic to a blowup of $\overline{M}_{0,5}$ at four points.} In part (4) above, we saw that if $m_1,\,m_2\geq 0>m_4$, then existence of three realizable and balanced cherries implies there are six exceptional divisors. So, we assume two are given by cherries $\Gamma_1=\langle 5\,||\,1,4\,|\,2,3\rangle$ and $\Gamma_1'$ whereas two by upside down cherries$\Gamma_2=\langle 2,5\,|\,3,4\,||\,1\rangle$ and $\Gamma_2'$. Equating the enhancements of $\Gamma_1$ and $\Gamma_2$, we obtain $m_1+m_4=m_2+m_3$ and $m_2+m_5=m_3+m_4$. Now, as in part (3), $\Gamma_1'$ has two options: $\langle 4\,||\,1,5\,|\,2,3\rangle$ or $\langle 5\,||\,2,4\,|\,1,3\rangle$. Suppose $\Gamma_1'=\langle 4\,||\,1,5\,|\,2,3\rangle$ then as in part (3) of the proof, we get $m_4=m_5$ (so $m_2+m_5=m_3+m_4$ implies $m_2=m_3$). Put $m_2=m_3=a$, so $m_1=m_2+m_3-m_4=2a-m_4$, so that $\sum_im_i=-2$ implies $m_4=-4a-2$ and thus $m_1=6a+2$, that is, $\mu=(6a+2,a^2,-4a-2^2)$, which is the tuple in part (5) of the proposition after a permutation and change of variable $a\rightarrow -a$.

  If instead $\Gamma_1'=\langle 5\,||\,2,4\,|\,1,3\rangle$, similar argument would yield $\mu=(4a-2^2,-a^2,-6a+2)$.
\item Since the presence of three exceptional divisors represented by cherry graphs (or, three by upside down cherry graphs) implies $\overline{B}$ has six exceptional divisors, there can be no $\mu$ for which $\overline{B}$ is isomorphic to a blowup of $\overline{M}_{0,5}$ at exactly five points. And we already saw in part (4) that if $\overline{B}$ is isomorphic to a blowup of $\overline{M}_{0,5}$ at six points then $\mu=(a^2,b^3)$ with $2a+3b=-2$.
  \end{enumerate}
\end{proof}
\section{Integral cohomology when $\mu=(0^{n-1},-2)$ or $\mu=(0^{n-2},-1^2)$}\label{integral_cohomology}
In the previous section, we saw that when $n\geq 7$, $\mu=(0^{n-1},-2)$ and $(0^{n-2},-1^2)$ are the only two cases where the coarse moduli space $\overline{B}$ is smooth. In this section, we aim to prove that in these cases, the \textit{integral} cohomology ring is generated by the boundary divisors. Since the proofs in both cases are essentially identical, we will assume $\mu=(0^{n-1},-2)$ for simplicity.

\begin{remark}\label{multiscaled_lines}
  In addition to being one of the few cases where the moduli space is a smooth variety, the case of $\mu=(0^{n-1},-2)$ is of independent interest. The space of multiscaled lines with collision $A_{n-1}$ considered in \cite{ar24} shares a lot of structural similarities with  $\overline{B}$. Both spaces aim to parametrize meromorphic differentials on a genus 0 curve with unique pole of order 2 at one point. The boundaries of both spaces are stratified by level graphs that are rooted level trees, and the number of levels corresponds to the codimension of the boundary stratum. However, there are two main differences: the points are allowed to collide in $A_{n-1}$, but not in $\overline{B}$, and the top level vertex carries a bona fide meromorphic differential in $A_{n-1}$ but carries a differential only up to a scalar factor in $\overline{B}$. Nevertheless, we expect there to be some relationships and interplay between these spaces, which have been further investigated in \cite{drz25} after the current work was completed and submitted.
\end{remark}

Because of the smoothness of the moduli space, it is reasonable to expect the integral version of  \cref{main-thm} to hold in this case. However, the property \ref{p6} in the Section \ref{ckgp_overview} (and the references cited therein) holds only for rational cohomology. So, we will take a completely different approach to computing the cohomology ring of $\overline{B}$ -- by using the birational morphism $\overline{B}\rightarrow \overline{M}_{0,n}$ and factoring the morphism as a sequence of blow ups along smooth, reduced and connected subschemes.

To realize this factorization, we use the ordering of the boundary divisors (really, only those exceptional over $\overline{M}_{0,n}$) given in \cite{ccm22}. Recall that the partial ordering on the set of boundary divisors is defined as follows: $D_1 <D_2$ if $D_1\cap D_2$ is non-empty, and $D_1\cap D_2$ is a degeneration of the top level of $D_2$ and the bottom level of $D_1$. Heuristically, this ordering says that if $D_1<D_2$ then the bottom level of $D_1$ ``admits more degenerations than the bottom level of $D_2$''. Equivalently, we can describe this ordering by saying $D_1<D_2$ if $D_1\cap D_2\neq \emptyset$ and the lower level stratum of $D_1$ has higher dimension than that of $D_2$. We then extend this partial order to a total order in an arbitrary way, and enumerate the exceptional divisors: $D_1<D_2<\ldots<D_m$. Then we claim that the birational morphism $\pi:\overline{B}\rightarrow \overline{M}_{0,n}$ factors as \[\overline{B}=X_m\rightarrow X_{m-1}\rightarrow\ldots\rightarrow X_0=\overline{M}_{0,n}\]such that the birational morphism $X_i\rightarrow X_{i-1}$ has (the proper transform of) $D_i$ as the exceptional divisor, but is biregular otherwise.\\

Note that, when $\mu=(0^{n-1},-2)$, every codimension 1 enhanced level graph will consist of a unique vertex on the top level, where the marked point of order $-2$ lies, and at most $\left\lfloor \frac{n-1}{2}\right\rfloor$ nodes on the bottom level (and the graph can have no horizontal edge). Let us denote by $Z_i$ the (reduced) image of $D_i$ in $\overline{M}_{0,n}$. Then we have:
\begin{claim}
  If $Z_i\subset Z_j$, then $i\leq j$.
\end{claim}
\begin{proof}
  Indeed, if $Z_i\subset Z_j$, then $D_i\cap D_j$ is represented by a three level graph whose underlying dual graph is the same as that for $Z_i$. Thus, the level graph for $D_i$ is obtained from that for $D_i\cap D_j$ via undegeneration keeping level 0 to -1 passage and collapsing the level -1 to -2 passage. This, by the definition of the partial order, implies $i\leq j$.
\end{proof}
For the first step in the above factorization, $X_1$ will be given by the blow up of $\overline{M}_{0,n}$ along $Z_1$. For this, we need to show that the preimage of $Z_1$ in $\overline{B}$ is a Cartier divisor (in fact, exactly $D_1$). So, take $\alpha\in Z_1$, then $\pi^{-1}(\alpha)\subset\overline{B}$ is connected by Zariski's main theorem. If $\pi^{-1}(\alpha)$ is not contained in $D_1$, then there is a component $F_0$ of the fiber not contained in $D_1$, but intersects $D_1$ non-trivially. A generic element of $F_0$ is associated to a level graph $\Delta$ that is not a degeneration of $\Gamma_1$ (the level graph for $D_1$), but the intersection $F_0\cap D_1$ will have generic element given by a level graph $\Lambda$ with same underlying dual graph as $\Delta$, such that as a level graph $\Lambda$ is a degeneration of $\Gamma_1$. Because of our ordering of boundary divisors, $\Lambda$ is obtained from $\Gamma_1$ by degenerating the lower level of $\Gamma_1$ only; since $\Lambda$ and $\Delta$ have the same top level (the root vertex), and undegenerating/consolidating all levels below the top of $\Lambda$ to one level yields $\Gamma_1$, same should be true for $\Delta$, thereby implying $F_0\subset D_1$, so that $\pi^{-1}(Z_1)=D_1$.

Since the preimage of $Z_1$ in $\overline{B}$ is a Cartier divisor, by the universal property of blow up, the morphism $\overline{B}\rightarrow\overline{M}_{0,n}$ factors through $X_1$. In the next step, we would like to blow up the proper transform of $Z_2$ in $X_1$ to obtain $X_2$. And indeed, this is possible -- since $Z_2$ is not contained in $Z_1$, its proper transform in $X_1$ is well defined, and since the preimage of $Z_2$ is again a Cartier divisor in $\overline{B}$, the morphism $\overline{B}\rightarrow X_1$ factors through $X_2\coloneqq \text{Bl}_{Z_2}(X_1)$ by the universal property of blow up (where for brevity, the proper transform of $Z_2$ in $X_1$ is also denoted by $Z_2$).

Continuing this way until we exhaust all the exceptional divisors in $\overline{B}$, we obtain in the end a smooth projective variety $X_m$ and a birational morphism $\overline{B}\rightarrow X_m$. But this birational morphism has no exceptional divisor, and smoothness of $X_m$ excludes the possibility of the exceptional locus being small (see \cite{km98}), thus the morphism has to be an isomorphism. Thus we have factorized the birational morphism $\overline{B}\rightarrow\overline{M}_{0,n}$ as a sequence of blow ups along a smooth connected scheme at each step.\\

The cohomology ring of a blowup can be described using Theorem 7.31 of \cite{voisin}:
\begin{theorem}[see e.g. \cite{voisin}, Thm. 7.31]\label{cohomology_blowup}
  Let $X$ be a K\"ahler manifold and $Z\subset X$ a submanifold. Denote by $\tau:\tilde{X}\rightarrow X$ the blow up of $X$ along $Z$ and by $E=\tau^{-1}(Z)$ the exceptional divisor (which is a projective bundle over $Z$). Let $h=c_1(\mathcal{O}_E(1))\in H^2(E,\mathbb{Z})$ be the first chern class of the tautological line bundle of $E$. Then we have isomorphism of the cohomology rings:
  \[H^k(X,\mathbb{Z})\oplus\left(\bigoplus_{i=0}^{r-2}H^{k-2i-2}(Z,\mathbb{Z})\right)\xrightarrow{\tau^*+\sum_ij_*\circ h^i\circ \tau|_E^*} H^k(\tilde{X},\mathbb{Z}),\]where $j:E\hookrightarrow\tilde{X}$ is the inclusion morphism.
\end{theorem}
So we have:
\begin{proposition}
  The ring $H^*(\mathbb{P}\Xi\overline{M}_{0,n}(-2,0^{n-1}),\mathbb{Z})$ is generated by the boundary divisors.
\end{proposition}
\begin{remark}
  The Chow ring and the cohomology ring of $\overline{B}$ are isomorphic by Theorem 2 of \cite[Appendix]{keel} (note that, at every step, we are blowing up the proper transform of a boundary stratum in $\overline{M}_{0,n}$, which remain ``Homology Isomorphism schemes'', using the terminology of \cite[Appendix]{keel}). In particular, all the odd cohomology groups vanish.
\end{remark}
\begin{proof}
  We apply \cref{cohomology_blowup} for $X=X_k,\ \tilde{X}=X_{k+1}$ and $Z=Z_k$ the center of the blow up for that step. Denote by $\tau:X_{k+1}\rightarrow X_k$ the birational morphism. Then, to apply the theorem, we need to prove that for any cohomology class $\alpha$ in $X_k$, its pullback $\tau^*\alpha$ is generated by the intersections of the boundary divisors and that for any cohomology class $\alpha$ in $Z_k$, the class $j_*\circ h^i\circ\tau|_E^*\alpha$ is generated by the intersections of the boundary divisors, where $E$ is the exceptional divisor of the blow up $\tau:X_{k+1}\rightarrow X_k$, $j$ is the inclusion of $E$ into $X_{k+1}$ and $h=c_1(\mathcal{O}_E(1))$ is the first chern class of the tautological bundle on $E$ (viewed as the projectivized normal bundle over $Z_k$). We will use induction on $k$ to prove this. As induction hypothesis, we will assume:
  \begin{enumerate}[(i)]
  \item The cohomology ring of $X_k$ is generated by boundary divisors (which include the exceptional divisors over $\overline{M}_{0,n}$ as well as the proper transform of the boundary divisors on $\overline{M}_{0,n}$),
  \item The boundary divisor of $X_k$ is a simple normal crossing divisor and the intersection of any collection of the irreducible components of the boundary is irreducible,
  \item For any intersection $Z$ of boundary divisors in $X_k$, the cohomology ring of $Z$ is also generated by boundary divisors (that is, the pullback of boundary divisors on $X_k$ to $Z$).
  \end{enumerate}
  The three statements are true for $\overline{M}_{0,n}$ (see \cite{keel} or \cite[Chap. XVII-7]{acg2}), so the base of induction is valid. Now we proceed with the inductive step. Take a cohomology class $\alpha$ represented by a codimension $p$ subvariety $W\subset X_k$ that is an intersection of boundary divisors in $X_k$. Write $W=D_1\cap \ldots \cap D_p$ for some boundary divisors $D_i\subset X_k$. First assume $W$ is not contained in $Z_k$. Then by the projection formula, the cohomology class $\tau^*\alpha$ differs from the cohomology class of the proper transform $\overline{W}$ of $W$ by a class supported in the exceptional divisor $E$ whose Gysin pushforward is 0. We have that $\overline{W}=\overline{D}_1\cap\ldots\cap \overline{D}_p$, where $\overline{D}_i$ is the proper transform of $D_i$, so to show cohomology ring of $X_{k+1}$ is also generated by boundary divisors, it is enough to show that the image of $j_*\circ h^i\circ\tau|_E^*$ is also generated by boundary divisors.
  
  To this end, we take a cohomology class $\alpha$ in $Z_k$ represented by the intersection $W$ of boundary divisors $D_1^k\cap\ldots\cap D_m^k\cap Z_k$. Without loss of generality, we can assume that none of $D_i^k$ contains $Z_k$ (or equivalently, if $Z_k=\Delta_1\cap\ldots\cap \Delta_r$ for boundary divisors $\Delta_i\subset X_k$ then none of $\Delta_i$ is equal to $D_j^k$); this particularly means that $D_i^k$ and $Z_k$ intersect transversally, so the preimage of $D_i^k$ in $X_{k+1}$ coincides with its proper transform $D_i^{k+1}$. Then since $\tau|_E$ is flat, the cohomology class $\tau|_E^*\alpha$ is represented by $\tau|_E^{-1}(W)$ which is equal to the intersection $D_1^{k+1}\cap\ldots\cap D_m^{k+1}\cap E$, that is, it is generated by boundary divisors. And since $h=-E|_E=j^*(-E)$, by projection formula, it follows that $j_*\circ h^i\circ\tau|_E^*\alpha$ is represented by $(-E)^i\cdot D_1^{k+1}\cdot\ldots\cdot D_m^{k+1}$, as desired. Thus, we have proven the statement (i) for $X_{k+1}$. The second statement (ii) is a standard fact about simple normal crossing divisor and blowups. The statement (iii) follows via argument similar to that for $X_{k+1}$ -- the map $\tau$ restricted to each such subvariety $Z$ is again a blow up of a smooth variety along smooth subvariety, and the situation is completely parallel. This completes the proof.
\end{proof}
If $\mu=(0^{n-2},-1^2)$ then also all the compatible divisorial level graphs that can appear are either a rooted level tree or horizontal. The horizontal boundary divisor is not exceptional over $\overline{M}_{0,n}$, so the exceptional divisors in $\overline{B}$ can still be given the same partial ordering as we did above, which can then be used to give a step by step blow up construction of $\overline{B}$ from $\overline{M}_{0,n+2}$ as above. So, the same argument as above yields
\begin{proposition}
  The ring $H^*(\mathbb{P}\Xi\overline{M}_{0,n}(0^{n-2},-1^2),\mathbb{Z})$ is generated by the boundary divisors.
\end{proposition}
\begin{remark}
  Even though we stated the preceding proposition only for $\mu=(0^{n-1},-2)$ and $(0^{n-2},-1^2)$, analogous ideas should be applicable for all the other smooth cases described in \cref{smooth_eq6} and \cref{smooth_nequals5} as well. In all those cases, the exceptional divisors were represeneted by either cherry graphs or upside down cherry graphs. In each of those cases, we can check that a cherry divisor and an upside down cherry divisor are always disjoint, so we can consider the same partial ordering on the set of cherry divisors as we did above and a different partial ordering on the upside down cherry divisor obtained instead by saying $D_1<D_2$ if the \textit{top} level of $D_1$ has higher dimension than the \textit{top} level of $D_2$. Because of the disjointness, we can blowup the centers of the cherry divisors first and then blowup the centers of the upside down cherry divisors in the order described above, and follow the same argument.
\end{remark}
\printbibliography
\end{document}